\documentclass[preprint,aos]{imsart}

\RequirePackage{amsthm,amsmath,amsfonts,amssymb}
\RequirePackage[numbers]{natbib}
\RequirePackage[colorlinks,citecolor=blue,urlcolor=blue]{hyperref}
\RequirePackage{graphicx} 
\startlocaldefs
\theoremstyle{plain}

\newtheorem{theorem}{Theorem}
\newtheorem{lemma}{Lemma}
\theoremstyle{remark}

\newtheorem{assumption}{Assumption}
\newtheorem{remark}{Remark}

\usepackage[utf8]{inputenc}
\usepackage[T1]{fontenc}
\usepackage[english]{babel}

\usepackage{tikz}
\usepackage{mathtools}
\usepackage{stmaryrd}
\usepackage{subcaption}
\usepackage{enumitem}

\usepackage{tikz}
\usetikzlibrary{matrix}
\usetikzlibrary{arrows}

\newcommand{\bbA}{{\bf A}}

\newcommand{\bbC}{{\bf C}}

\newcommand{\bbH}{{\bf H}}
\newcommand{\bbS}{{\bf S}}

\newcommand{\bbT}{{\bf T}}
\newcommand{\bbL}{{\bf L}}
\newcommand{\bbI}{{\bf I}}
\newcommand{\bbM}{{\bf M}}

\newcommand{\bbB}{{\bf B}}

\newcommand{\bbP}{{\bf P}}
\newcommand{\bbQ}{{\bf Q}}

\newcommand{\bbX}{{\bf X}}
\newcommand{\bbx}{{\bf x}}

\newcommand{\bbY}{{\bf Y}}
\newcommand{\bby}{{\bf y}}

\newcommand{\bbW}{{\bf W}}
\newcommand{\bbU}{{\bf U}}
\newcommand{\bbV}{{\bf V}}

\newcommand{\tr}{{\rm tr}}
\newcommand{\Tr}{{\rm Tr}}

\endlocaldefs

\begin{document}

\begin{frontmatter}
\title{The limiting spectral distribution of the sample canonical correlation matrix}
\runtitle{The limiting spectral distribution of the sample canonical correlation matrix}

\begin{aug}
\author[A]{\fnms{Xiaozhuo}~\snm{Zhang}\ead[label=e1]{zhangxz722@nenu.edu.cn}}
\address[A]{Key Laboratory of Applied Statistics of MOE\, School of Mathematics and Statistics\, 
Northeast Normal University\, Changchun, Jilin 130024, China\printead[presep={,\ }]{e1}}

\end{aug}

\begin{abstract}
	In this paper, we investigate the spectral properties of the sample canonical correlation (SCC) matrix under the alternative hypothesis to provide a more comprehensive description of the association between two sets of variables. Our research involves establishing the relationship between the eigenvalues of the SCC matrix and the block correlation matrix, as well as proving the universality of the Stieltjes transform of the limiting spectral distribution (LSD) of the block correlation matrix. By combining the results from the normal case, we establish the limiting spectral distribution (LSD) of the SCC matrix with a general underlying distribution under the arbitrary rank alternative hypothesis. Finally, we present several simulated examples and find that they fit well with our theoretical results.
\end{abstract}

\begin{keyword}[class=MSC]
\kwd[Primary ]{60B10}
\kwd[; secondary ]{62H10}
\end{keyword}

\begin{keyword}
\kwd{Canonical correlation analysis}
\kwd{Block correlation matrix}
\kwd{Limiting spectral distribution}
\end{keyword}

\end{frontmatter}


\section{Model and Relationship}	
When we encounter multivariables or high dimensional data, principal component analysis as a traditional research method for data reduction and interpretation, fought against data explosion by focusing on the significant information reflected in original data.
However, when confronted with more complex problems in practical applications such as statistical learning, wireless communications, financial economics and population genetics, Canonical Correlation Analysis (CCA) was proposed to identify and quantify the associations between two sets of variables.

Traditional multivariate analysis textbooks \cite{hotelling1935most,hotelling1936rela,Mu1982,Anderson2003} use matrix language to deconstruct the CCA problem and draw a heuristic conclusion: the key quantity that reflects the correlation between two groups of variables is the characteristic space of the canonical correlation matrix. Specifically, for $p+q$ dimensional general population $(\tilde{\bbx}',\tilde{\bby}')'$ with mean $\bf 0$, the covariance between $\tilde{\bbx}$ and $\tilde{\bby}$ has the following expression
\begin{align*}
\mathbb{V}ar\begin{pmatrix} \tilde{\bbx}\\   \tilde{\bby}\end{pmatrix}=\begin{pmatrix}
		\tilde{\boldsymbol{\Sigma}}_{xx}, \tilde{\boldsymbol{\Sigma}}_{xy}\\
		\tilde{\boldsymbol{\Sigma}}_{yx}, \tilde{\boldsymbol{\Sigma}}_{yy}
	\end{pmatrix}.
\end{align*}
The aimed canonical correlation coefficients can be characterized as the square roots of the eigenvalues of the population canonical correlation (PCC) matrix $\mathcal{C}_{xy}=\tilde{\mathbf\Sigma}_{xx}^{-1/2}\tilde{\mathbf\Sigma}_{xy}\tilde{\mathbf\Sigma}_{yy}^{-1}\tilde{\mathbf\Sigma}_{yx}\tilde{\mathbf\Sigma}_{xx}^{-1/2}$. According to scale invariance, we can simply assume that the covariance matrix of the population $(\bbx',\bby')'$ with mean $\bf 0$ has the following structure:
\begin{align*}
\mathbb{V}ar\begin{pmatrix} \bbx\\   \bby\end{pmatrix}=\begin{pmatrix}
			\bbI_p,&\boldsymbol{\Lambda}\\
		\boldsymbol{\Lambda},&\bbI_q	\end{pmatrix}.
\end{align*}
For sample matrices $\bbX_{p\times n}$ and $\bbY_{q\times n}$ from population $\bbx$ and $\bby$ respectively, write $\bbS_{xx}=\frac{1}{n}\bbX\bbX'$, $\bbS_{yy}=\frac{1}{n}\bbY\bbY'$, and $\bbS_{xy}=\frac{1}{n}\bbX\bbY'$.
Then, we define the sample canonical correlation (SCC) matrix as $\mathcal{S}_{xy}=\bbS_{xx}^{-1/2}\bbS_{xy}\bbS_{yy}^{-1}\bbS_{yx}\bbS_{xx}^{-1/2}$ to spy on the information of population canonical correlations.

We pause to review some important studies on the spectrum of the SCC matrix.
Under the null hypothesis, i.e., $\bbx$ and $\bby$ are independent, Yang and Pan established the LSD of SCC matrix \cite{2012The} and the central limiting theorem(CLT) for linear spectral statistics (LSS) \cite{2015Independence}. Under finite rank case of alternative hypothesis, i.e., the fixed rank of $\boldsymbol{\Lambda}$, Bao et al. \cite{BaoH19Canonical} gave the asymptotic distribution of spiked eigenvalues and extreme eigenvalues with Gaussian entries, then Yang and Ma showed the universality of the above results in \cite{yang2022limiting,ma2021sample}.  
For the general alternative hypothesis, i.e., the arbitrary rank of $\boldsymbol{\Lambda}$, Bai et al. \cite{BHHJZ2020} proposed a new way to transform the eigenvalues of the SCC matrix to the function of that of the noncentral Fisher matrix under Gaussian assumption. Then under the same condition, Hou et al. \cite{hou2023spiked} gave the CLT for spiked eigenvalues of the SCC matrix via the corresponding results of noncentral Fisher matrix. Now we face the problem of asymptotic results for the SCC matrix under alternative hypothesis and general underlying distribution. 

Noticed it is hard to extend the results to the data matrices with generally distributed entries using the idea in \cite{BHHJZ2020}, because it depended crucially on the equivalence of uncorrelation and independence of normal distribution. Essentially, the goal of CCA is to seek the relationship between the block $\bbx$ and block $\bby$ in random vector $(\bbx',\bby')'$, which is the same as that of the block correlation matrix. Block correlation matrix dated back to Bao et al. \cite{Bao2017test}, which was developed by sample correlation matrix.  Recently the results \cite{bao2022spectral} about this concept are expanded to general cases, for example, no matter  whether $k$ is fixed or divergent, more mild conditions for dimensional size ratio.

Our approach adapts the same strategy as in \cite{2012The}. The key innovation is to find the relationship between the SCC matrix and block correlation matrix and to get the exact expression for Stieltjes transform of the LSD of the SCC matrix. Now we focus on case $k=2$ to calculate the relationship between the SCC matrix and block correlation matrix. We construct the block correlation matrix $\bbB$ for sample matrices $\bbX_{p\times n}$ and $\bbY_{q\times n}$, where
\begin{align*}
	\bbB=&\begin{pmatrix}
		(\bbX\bbX')^{-\frac{1}{2}}, &\boldsymbol{0}\\
		\boldsymbol{0}, &(\bbY\bbY')^{-\frac{1}{2}}
	\end{pmatrix}
	\begin{pmatrix}
		\bbX\bbX', &\bbX\bbY'\\
		\bbY\bbX', &\bbY\bbY'
	\end{pmatrix}
	\begin{pmatrix}
		(\bbX\bbX')^{-\frac{1}{2}}, &\boldsymbol{0}\\
		\boldsymbol{0}, &(\bbY\bbY')^{-\frac{1}{2}}
	\end{pmatrix}\\
=&\begin{pmatrix}
		\bbI_p,&(\bbX\bbX')^{-\frac{1}{2}}\bbX\bbY'(\bbY\bbY')^{-\frac{1}{2}}\\
		(\bbY\bbY')^{-\frac{1}{2}}\bbY\bbX'(\bbX\bbX')^{-\frac{1}{2}},&\bbI_q
	\end{pmatrix}.
\end{align*}
The block correlation matrix presented above is a natural extension of the sample correlation matrix in block form.
Considering the Green function $G_b(z)=(\bbB-z)^{-1}$ of $\bbB$, we have
\begin{align*}
	&G_b(z)=\left[\begin{pmatrix}
		\bbI_p,&(\bbX\bbX')^{-\frac{1}{2}}\bbX\bbY'(\bbY\bbY')^{-\frac{1}{2}}\\
		(\bbY\bbY')^{-\frac{1}{2}}\bbY\bbX'(\bbX\bbX')^{-\frac{1}{2}},&\bbI_q
	\end{pmatrix}-\begin{pmatrix}
		z\bbI_p,&\boldsymbol{0}\\
		\boldsymbol{0},&z\bbI_q
	\end{pmatrix}\right]^{-1}\\
	=&\begin{pmatrix}
		(1-z)\bbI_p,&(\bbX\bbX')^{-\frac{1}{2}}\bbX\bbY'(\bbY\bbY')^{-\frac{1}{2}}\\
		(\bbY\bbY')^{-\frac{1}{2}}\bbY\bbX'(\bbX\bbX')^{-\frac{1}{2}},&(1-z)\bbI_q
	\end{pmatrix}^{-1}\\
	=&\begin{pmatrix}
		[(1\!-\!z)\bbI_p\!-\!\frac{1}{1\!-\!z}(\bbX\bbX')^{-\frac{1}{2}}\bbX\bbP_y\bbX'(\bbX\bbX')^{-\frac{1}{2}}]^{-1},& \ast \\
		\ast,&[(1\!-\!z)\bbI_q\!-\!\frac{1}{1\!-\!z}(\bbY\bbY')^{-\frac{1}{2}}\bbY\bbP_{x}\bbY'(\bbY\bbY')^{-\frac{1}{2}}]^{-1}
	\end{pmatrix}\\
	=&(1-z)\begin{pmatrix}
		[(1\!-\!z)^{2}\bbI_p\!-\!(\bbX\bbX')^{-\frac{1}{2}}\bbX\bbP_y\bbX'(\bbX\bbX')^{-\frac{1}{2}}]^{-1},& \ast\ast \\
		\ast\ast,&[(1\!-\!z)^2\bbI_q\!-\!(\bbY\bbY')^{-\frac{1}{2}}\bbY\bbP_x\bbY'(\bbY\bbY')^{-\frac{1}{2}}]^{-1}
	\end{pmatrix}.
\end{align*}
Taking trace for the matrix on both sides of the equation, that is 
\begin{align*}
	\Tr G_b=(z-1)[\Tr G_{xy}((1-z)^2)+\Tr G_{yx}((1-z)^2)],
\end{align*}
where the function $\frac{1}{p}\Tr G_{xy}(z):=\frac{1}{p}\Tr(\mathcal{S}_{xy}-z)^{-1}$  is the Stieltjes transform of the SCC matrix, and $\frac{1}{q}\Tr G_{yx}(z):=\frac{1}{q}\Tr(\mathcal{S}_{yx}-z)^{-1}$ is similar. Notice that $\mathcal{S}_{xy}$ and $\mathcal{S}_{yx}$ share the same non-zero eigenvalues, so the limit of $\frac{1}{q}\Tr G_{xy}$ under the alternative hypothesis can be converted to the limit of $\frac{1}{p+q}\Tr G_b$. This implies the spectral distributions of $\bbB$ and the SCC matrix have some relationship at a global level. 

Considering the correlation coefficients matrix  $\boldsymbol{\Lambda}$, we can assume the entries of $\bbW_{p\times n}$ and $\bbY_{q\times n}$ are i.i.d., and 
\begin{align}\label{xzassume}
	\bbX_{p\times n}=\boldsymbol{\Lambda}_{p\times q}\bbY_{q\times n}+\boldsymbol{\Gamma}\bbW_{p\times n},\ \bbY_{q\times n}=\bbY_{q\times n}
\end{align}
where $\boldsymbol{\Gamma}\boldsymbol{\Gamma}'=\bbI_p-\boldsymbol{\Lambda}\boldsymbol{\Lambda}'$.
Because of the scale invariance in form \eqref{xzassume}, we replace $\bbX,\bbY,\bbW$ in \eqref{xzassume} with the $\frac{1}{\sqrt{n}}\bbX,\frac{1}{\sqrt{n}}\bbY, \frac{1}{\sqrt{n}}\bbW$,  but still expressed as $\bbX,\bbY,\bbW$, respectively.

The paper is organized as follows. In Section 2, we provide the main results of the paper. Sections 3-5 show the important stages in the proof of Theorem \ref{th12}. The lemmas appearing in previous sections and their proof are stated in Section 6.

\section{Main Result}
\begin{theorem}\label{th11}
Denote the ordered non-zero eigenvalues of $\mathcal{S}_{xy}$ by $l_1^2\geq l_2^2\geq \dots \geq l_{\min\{p,q\}}^2$, then the $p+q$ eigenvalues of $\bbB$ will consist of 
\begin{align*}
\{\underbrace{1+l_1,\dots,1+l_{\min\{p,q\}}}_{\min\{p,q\}},\underbrace{1,\dots,1}_{\max\{p,q\}-\min\{p,q\}},\underbrace{1-l_{\min\{p,q\}},\dots,1-l_1}_{\min\{p,q\}}\}.
\end{align*} 
Moreover, at the global level, we have 
\begin{align*}
	\Tr G_b=(z-1)[\Tr G_{xy}((1-z)^2)+\Tr G_{yx}((1-z)^2)].
\end{align*}
\end{theorem}

\begin{remark}
We establish a general result in Lemma \ref{lemma0}. The proof of Theorem \ref{th11} is a specific instance where we select the square root matrix of $\mathcal{S}_{xy}$ as $\bbA$. By combining this with the argument for $\Tr G_b$ in Section 1, we successfully complete the proof. 
\end{remark}

\begin{assumption}[Assumption on the dimensions]\label{assum1}
Assume that $p/n\to c_1$ and $q/n\to c_2$ as $n\to\infty$, such that $c_1+c_2<1$.  Meanwhile, we always work with the additional conditions: $p/q\to c_3<1$.
\end{assumption}
\begin{assumption}[Assumption on matrix entries]\label{assum2}
	Assume that real matrices $\bbW_{p\times n}$ and $\bbY_{q\times n}$ in \eqref{xzassume} have i.i.d. entries. Furthermore, for all $a \in [\![ p ]\!]$ and $b \in [\![ n ]\!]$, the distribution of $W_{ab}$ belong distribution set $\mathcal{F}$, and for all $a \in [\![ q ]\!]$ and $b \in [\![ n ]\!]$, the distribution of $Y_{ab}$ also belong distribution set $\mathcal{F}$, where 
	\begin{align*}
		\mathcal{F}=\{F_y:\mathbb{E}[z]=0, \mathbb{E}[|z|^2]= 1, \ \mbox{and}\  \mathbb{E}[|z|^4]<\infty\}.
	\end{align*}
\end{assumption}
\begin{assumption}[Assumption on $\boldsymbol{\Lambda}$]\label{assum3}
We assume the operator norm of $\boldsymbol{\Lambda}$ is bounded and keep away from $1$.  So we conclude the operator norm of $\boldsymbol{\Gamma}$ is bounded and away from $0$. The empirical distribution the singular values of $\boldsymbol{\Lambda}$ converges weakly to a non random distribution $H$.
\end{assumption}

\begin{theorem}\label{th12}
	Under Assumptions \ref{assum1}-\ref{assum3}, we have the LSDs of the SCC matrix and its corresponding block correlation matrix are universal. Then Stieltjes transform of the LSD of the SCC matrix is the unique solution of the following equation set:
	\begin{align}\label{CCALSDs}
	G_1(z,m)&\frac{c_2(1-z)[(1-z)m-1]}{1-c_2+c_1[(1-z)m-1]}=\int\left[\frac{n}{q}\frac{x^2}{1-x^2}-\frac{G_2(z,m)}{G_1(z,m)}\right]^{-1}dH(x)\\
	G_1(z,m)=&1-c_1-c_1[(1-z)m-1]+\frac{c_1(1-c_2)(1-z)[(1-z)m-1]}{1-c_2+c_1z[(1-z)m-1]}\nonumber\\
	G_2(z,m)=&\frac{1-c_1-c_2+c_1(1-z)m}{c_2(1-z)}\left[1-\frac{(1-z)(1-c_1)}{1-c_2-c_1z+c_1zm(1-z)}\right]\nonumber .
    \end{align}
\end{theorem}
\begin{proof}
Here we roughly introduce the proof strategy. Under normal case, we can transform the SCC matrix to noncentral Fisher matrix, and through the study of the noncentral Fisher matrix, we have the corresponding results for the SCC matrix. Intuitively, the result of the LSD should be universality, now the restriction is that the relationship does not work under non-Gaussian case. We choose another transformation which holds under the general case, that is, the block correlation matrix under some appropriate alternative hypothesis. Given the result of the normality in \cite{BHHJZ2020}, we only need the universality of the LSD of the SCC matrix, that is, it is enough to consider that of the corresponding block correlation matrix $\bbB$. Notice the non-zero eigenvalues of $\bbB$ equal that of $\bbH$, where 
\begin{align*}
	\bbH=\bbX'(\bbX\bbX')^{-1}\bbX+\bbY'(\bbY\bbY')^{-1}\bbY:=\bbP_x+\bbP_y.
\end{align*}
So in subsequent sections, we mainly prove that the LSD of $\bbH$ is universal.
\end{proof}

	\section{Expectation convergence}\label{s3}
	In this section, we will show the difference of $\tr(\bbH-z)^{-1}$ ($\tr$ means normalized trace as $\frac{1}{n}\Tr$) between Gaussian and general case in the sense of expectation. There are two random parts contained in matrix $\bbH$, then we substitute them step by step. Here the substitute method is not Lindeberg’s method but another continuous method widely used in \cite{2015Independence, knowles2017anisotropic}.
	Write $\bbW(s)=s^{1/2}\bbW+(1-s)^{1/2}\hat{\bbW}$, where the matrix $\hat{\bbW}$ with standard normal entries is independent with $\bbW$. Let $\bbH_N$ and $\bbH(s)$ denote the matrix in which $\bbW$ is replaced by $\hat{\bbW}$ and $\bbW(s)$ respectively, but $\bbY$ is still under general distribution. Then we have 
	\begin{align*}
		&\mathbb{E}\tr(\bbH-z)^{-1}-\mathbb{E}\tr(\bbH_N-z)^{-1}=\int_0^1\frac{\partial}{\partial s}\mathbb{E}\tr(\bbH(s)-z)^{-1}ds\\
		=&-\int_0^1\mathbb{E}\tr(\bbH(s)-z)^{-2}\frac{\partial \bbP_{x(s)}}{\partial s}ds.
	\end{align*}
Similar to the derivative of a matrix by a scalar, we have 
	\begin{align*}
		&\frac{\partial\bbP_{x(s)}}{\partial s}=\frac{\partial [\bbX(s)]'\left(\bbX(s)[\bbX(s)]'\right)^{-1}\bbX(s)}{\partial s}\\
		=&\frac{1}{2}(\bbI-\bbP_{x(s)})\left(\frac{\bbW}{\sqrt{s}}\!-\!\frac{\hat{\bbW}}{\sqrt{1\!-\!s}}\right)'\boldsymbol{\Gamma}'\bbQ(s)+\frac{1}{2}\bbQ(s)'\boldsymbol{\Gamma}\left(\frac{\bbW}{\sqrt{s}}\!-\!\frac{\hat{\bbW}}{\sqrt{1\!-\!s}}\right)(\bbI-\bbP_{x(s)})
	\end{align*}
	where by definition $\bbQ(s)=(\bbX(s)\bbX(s)')^{-1}\bbX(s)$.
	Then we have
	\begin{align*}
		&\mathbb{E}\tr(\bbH(s)-z)^{-2}\frac{\partial \bbP_{x(s)}}{\partial s}\\
		=&\mathbb{E}\tr\left(\frac{\bbW}{\sqrt{s}}\!-\!\frac{\hat{\bbW}}{\sqrt{1\!-\!s}}\right)'\boldsymbol{\Gamma}'\bbQ(s)(\bbH(s)-z)^{-2}(\bbI-\bbP_{x(s)})\\
		=&\frac{1}{n}\mathbb{E}\left[\frac{1}{\sqrt{s}}\sum_{ij}W_{ji}\Phi_{ji}-\frac{1}{\sqrt{1-s}}\sum_{ij}\hat{W}_{ji}\Phi_{ji}\right]
	\end{align*}
	where 
	\begin{align}\label{phi}
		\Phi_{ji}:=[\boldsymbol{\Gamma}'\bbQ(s)(\bbH(s)-z)^{-2}(\bbI-\bbP_{x(s)})]_{ji}.
	\end{align}
	So 
	\begin{align*}
		\mathbb{E}\tr(\bbH\!-\!z)^{-1}\!-\!\mathbb{E}\tr(\bbH_N\!-\!z)^{-1}=\!-\!\int_0^1\frac{1}{n}\mathbb{E}\left[\frac{1}{\sqrt{s}}\sum_{ij}W_{ji}\Phi_{ji}\!-\!\frac{1}{\sqrt{1\!-\!s}}\sum_{ij}\hat{W}_{ji}\Phi_{ji}\right]ds.
	\end{align*}
	
	Then we need to calculate $\mathbb{E}\frac{1}{\sqrt{s}}\sum_{ij}W_{ji}\Phi_{ji}$ and $\mathbb{E}\frac{1}{\sqrt{1-s}}\sum_{ij}\hat{W}_{ji}\Phi_{ji}$ respectively. Here $\hat{W}_{ji}$ is from  standard normal distribution, so 
	\begin{align*}
		\mathbb{E}\sum_{ij}\hat{W}_{ji}\Phi_{ji}=\frac{1}{n}\sum_{ij}\mathbb{E}\frac{\partial \Phi_{ji}}{\partial \hat{W}_{ji}}
	\end{align*}
	exactly holds. By performing basic matrix calculus, we conclude the term with $\kappa_2$ in $\mathbb{E}\frac{1}{\sqrt{s}}\sum_{ij}W_{ji}\Phi_{ji}$ and $\mathbb{E}\frac{1}{\sqrt{1-s}}\sum_{ij}\hat{W}_{ji}\Phi_{ji}$ are exactly same, see more details for Lemma \ref{lemma1}. So the terms with third and higher order cumulants of $\mathbb{E}\frac{1}{\sqrt{s}}\sum_{ij}W_{ji}\Phi_{ji}$ are remained in their difference, next we bound it via the following formulas: 
	\begin{align*}
		&\frac{1}{n}\mathbb{E}\left[\frac{1}{\sqrt{s}}\sum_{ij}W_{ji}\Phi_{ji}-\frac{1}{\sqrt{1-s}}\sum_{ij}\hat{W}_{ji}\Phi_{ji}\right]=\frac{1}{n}\frac{1}{\sqrt{s}}\sum_{ij}\mathbb{E}e_1
	\end{align*}
	where the term $e_1$ results from cumulant expansion and  
	\begin{align*}
		e_1=O(1)\mathbb{E}[|W_{ji}|^3I_{\{|W_{ji}|>t\}}]\sup_{W_{ji}\in\mathbb{R}}\left|\frac{\partial^2 \Phi_{ji}}{\partial W^2_{ji}}\right|+O(1)\mathbb{E}|W_{ji}|^3\sup_{|W_{ji}|<t}\left|\frac{\partial^2 \Phi_{ji}}{\partial W^2_{ji}}\right|.
	\end{align*}
	Note that all terms in $\frac{\partial^2 \Phi_{ji}}{\partial W^2_{ji}}$ can be expressed in terms of the following factors
\begin{align*}
	&[\boldsymbol{\Gamma}'(\bbX(s)\bbX(s)')^{-1}\boldsymbol{\Gamma}]_{jj},[\boldsymbol{\Gamma}'\bbQ(s)G^b\bbQ'(s)\boldsymbol{\Gamma}]_{jj},[G^b]_{ii},[\bbP_{x(s)}G^b]_{ii},\\
	&[\bbP_{x(s)}G^b\bbP_{x(s)}]_{ii},[\boldsymbol{\Gamma}'\bbQ(s)G^b]_{ji}, [\boldsymbol{\Gamma}'\bbQ(s)G^b\bbP_{x(s)}]_{ji}
\end{align*}
where $b=0,1,2$ and $G=(\bbH(s)-z)^{-1}$ (we omit $s$ without aumbiguity).

According to the Assumption \ref{assum2}, we have $\|\boldsymbol{\Gamma}\|$ is bounded. For any fixed $z\in\mathbb{C}^{+}$, $\|G(z)\|$ has the trivial bound $\frac{1}{v}$. Here we need Lemma \ref{lemma2} to bound $\|\bbQ\|$ with high probability.  Meanwhile from the moment assumption on matrix entries, we have
\begin{align*}
	 \mathbb{E}[|W_{ji}|^3I_{\{|W_{ji}|>t\}}]=o\left(\frac{1}{n^{3/2}}\right) \mbox{ and }   \mathbb{E}|W_{ji}|^3=O\left(\frac{1}{n^{3/2}}\right) .
\end{align*}
Considering there is at least one off-diagonal element in term $\frac{\partial^2 \Phi_{ji}}{\partial W^2_{ji}}$, and choosing one of off-diagonal element denoted $A_{ij}$, we have 
	\begin{align*}
		&\frac{1}{n}\frac{1}{\sqrt{s}}\sum_{ij}\mathbb{E}[|W_{ji}|^3I_{\{|W_{ji}|>t\}}]\sup_{W_{ji}\in\mathbb{R}}\left|\frac{\partial^2 \Phi_{ji}}{\partial W^2_{ji}}\right|=\mathbb{E}\frac{1}{n}\sum_{ij}|A_{ij}|o(\frac{1}{n^{3/2}})\\
		\leq&\frac{1}{n}\sqrt{\mathbb{E}\sum_{ij}|A_{ij}|^2}\sqrt{np} \times o(\frac{1}{n^{3/2}})=o(\frac{1}{n}),\\
	\mbox{and  }	&\frac{1}{n}\frac{1}{\sqrt{s}}\sum_{ij}\mathbb{E}|W_{ji}|^3\sup_{|W_{ji}|<t}\left|\frac{\partial^2 \Phi_{ji}}{\partial W^2_{ji}}\right|=\mathbb{E}\frac{1}{n}\sum_{ij}|A_{ij}|O(\frac{1}{n^{3/2}})\\
		\leq&\frac{1}{n}\sqrt{\mathbb{E}\sum_{ij}|A_{ij}|^2}\sqrt{np} \times O(\frac{1}{n^{3/2}})=O(\frac{1}{n}).
	\end{align*}
	This implies $\mathbb{E}\tr(\bbH-z)^{-1}-\mathbb{E}\tr(\bbH_N-z)^{-1}=O(\frac{1}{n})$.
	Now we need to confirm that after replacing $\bbY$ with $\hat{\bbY}$, whose elements are with standard normal distribution, the limit of the quality $\mathbb{E}\tr(\bbH-z)^{-1}$ does not change, i.e., 
	\begin{align*}
		&\mathbb{E}\tr(\bbH-z)^{-1}-\mathbb{E}\tr(\bbH_N-z)^{-1}=-\int_0^1O(\frac{1}{n})ds=O(\frac{1}{n}),
	\end{align*}
	where the notion $\bbH$ is the matrix in which $\bbW$ with Gaussian case and $\bbY$ with general distribution. Similar to the previous interpolation step, we have
	\begin{align*}
		&\mathbb{E}\tr(\bbH-z)^{-1}-\mathbb{E}\tr(\bbH_N-z)^{-1}=\int_0^1\frac{\partial}{\partial t}\mathbb{E}\tr(\bbH(t)-z)^{-1}dt\\
		=&-\int_0^1\mathbb{E}\tr(\bbH(t)-z)^{-2}\left[\frac{\partial \bbP_{x(t)}}{\partial t}+\frac{\partial \bbP_{y(t)}}{\partial t}\right]dt.
	\end{align*}
Set
	\begin{align*}
		\bbY(t)=\sqrt{t}\bbY+\sqrt{1-t}\hat{\bbY},\ 	\bbX(t)=\boldsymbol{\Lambda}\bbY(t)+\boldsymbol{\Gamma}\hat{\bbW}, \  \bbU(t)=(\bbY(t)\bbY'(t))^{-1}\bbY(t),
	\end{align*}
	by the derivative calculation of a matrix by a scalar, we have 
	\begin{align*}
		&\frac{\partial\bbP_{x(t)}}{\partial t}=\frac{\partial \bbX'(t)(\bbX(t)\bbX'(t))^{-1}\bbX(t)}{\partial t}\\
		=&(\bbI\!-\!\bbP_{x(t)})\frac{1}{2}\left(\frac{1}{\sqrt{t}}\boldsymbol{\Lambda}\bbY\!-\!\frac{1}{\sqrt{1\!-\!t}}\boldsymbol{\Lambda}\hat{\bbY}\right)'\bbQ(t)\!+\!\frac{1}{2}\bbQ(t)'\left(\frac{1}{\sqrt{t}}\boldsymbol{\Lambda}\bbY\!-\!\frac{1}{\sqrt{1\!-\!t}}\boldsymbol{\Lambda}\hat{\bbY}\right)(\bbI\!-\!\bbP_{x(t)})
	\end{align*}
	and
	\begin{align*}
		&\frac{\partial\bbP_{y(t)}}{\partial t}=\frac{\partial \bbY'(t)(\bbY(t)\bbY'(t))^{-1}\bbY(t)}{\partial t}\\
		=&\frac{1}{2}(\bbI-\bbP_{y(t)})\left(\frac{\bbY}{\sqrt{t}}\!-\!\frac{\hat{\bbY}}{\sqrt{1\!-\!t}}\right)'\bbU(t)+\frac{1}{2}\bbU(t)'\left(\frac{\bbY}{\sqrt{t}}\!-\!\frac{\hat{\bbY}}{\sqrt{1\!-\!t}}\right)(\bbI-\bbP_{y(t)}).
	\end{align*}
	Then 
	\begin{align*}
		\mathbb{E}\tr(\bbH(t)-zI)^{-2}\left[\frac{\partial \bbP_{x(t)}}{\partial t}+\frac{\partial \bbP_{y(t)}}{\partial t}\right]
		=\frac{1}{n}\sum_{l=1}^2\mathbb{E}\left[\frac{1}{\sqrt{t}}\sum_{ij}Y_{ji}\Psi_{ji}^{l}\!-\!\frac{1}{\sqrt{1-t}}\sum_{ij}\hat{Y}_{ji}\Psi_{ji}^{l}\right]
	\end{align*}
	holds, where for $l=1,2$, $\Psi_{ji}^{1}=[\boldsymbol{\Lambda}'\bbQ(t)G^2(\bbI-\bbP_{x(t)})]_{ji},\ \Psi_{ji}^{2}=[\bbU(t)G^2(\bbI-\bbP_{y(t)})]_{ji}.$
	Similarly, we need to calculate $\mathbb{E}\frac{1}{\sqrt{t}}\sum_{ij}Y_{ji}\Psi^l_{ji}$ and $\mathbb{E}\frac{1}{\sqrt{1-t}}\sum_{ij}\hat{Y}_{ji}\Psi^l_{ji}$ respectively. Here $\hat{Y}_{ji}$ is standard normal distribution, so we have exactly 
	\begin{align*}
		\mathbb{E}\sum_{ij}\hat{Y}_{ji}\Psi^l_{ji}=\frac{1}{n}\sum_{ij}\mathbb{E}\frac{\partial \Psi^l_{ji}}{\partial \hat{Y}_{ji}}.
	\end{align*}
From the arguments in Lemma \ref{lemma3}, we conclude the term with $\kappa_2$ in $\mathbb{E}\frac{1}{\sqrt{t}}\sum_{ij}Y_{ji}\Psi^l_{ji}$ and $\mathbb{E}\frac{1}{\sqrt{1-t}}\sum_{ij}\hat{Y}_{ji}\Psi^l_{ji}$ are exactly same for $l=1,2$. So we need to calculate and bound the terms with third and later cumulants of $\mathbb{E}\frac{1}{\sqrt{t}}\sum_{ij}Y_{ji}\Psi^l_{ji}$:
	\begin{align*}
		&\sum_{l=1}^2\frac{1}{n}\mathbb{E}\left[\frac{1}{\sqrt{t}}\sum_{ij}Y_{ji}\Psi^l_{ji}-\frac{1}{\sqrt{1\!-\!t}}\sum_{ij}\hat{Y}_{ji}\Psi^l_{ji}\right]=\frac{1}{n}\frac{1}{\sqrt{t}}\sum_{l=1}^2\sum_{ij}\mathbb{E}e_2^l,\\
		&e_2^l=O(1)\mathbb{E}[|Y_{ji}|^3I_{\{|Y_{ji}|>t\}}]\sup_{Y_{ji}\in\mathbb{R}}\left|\frac{\partial^2 \Psi^l_{ji}}{\partial Y^2_{ji}}\right|+O(1)\mathbb{E}|Y_{ji}|^3\sup_{|Y_{ji}|<t}\left|\frac{\partial^2 \Psi^l_{ji}}{\partial Y^2_{ji}}\right|,
	\end{align*}
	where the term  $e_2^l$  results from cumulant expansion.

	We notice that the bound of $\left|\frac{\partial^2 \Psi^l_{ji}}{\partial Y^2_{ji}}\right|$ is only related with its property (at least one off-diagonal element factor), not with the specific form. Similarly,
	$$\frac{1}{n}\frac{1}{\sqrt{t}}\sum_{l=1}^2\sum_{ij}\mathbb{E}e_2^l=O(\frac{1}{n}).$$
	
	\section{Martingale difference decomposition}\label{s4}
	Define $\mathbb{E}_k=\mathbb{E}(\cdot|w_1,y_1,\dots,w_k,y_k)$, then we have 
	\begin{align*}
		&\tr (\bbH-z)^{-1}-\mathbb{E}\tr (\bbH-z)^{-1}=\sum_{k=1}^n(\mathbb{E}_k-\mathbb{E}_{k-1})\tr (\bbH-z)^{-1}\\
		=&\sum_{k=1}^n(\mathbb{E}_k-\mathbb{E}_{k-1})[\tr (\bbH-z)^{-1}-\tr (\bbH^{(k)}-z)^{-1}],
	\end{align*}
	where 
	\begin{align*}
		\bbH^{(k)}=\bbP_{x}^{(k)}+\bbP_{y}^{(k)}=(\bbX^{(k)})'[\bbX^{(k)}(\bbX^{(k)})']^{-1}\bbX^{(k)}+(\bbY^{(k)})'[\bbY^{(k)}(\bbY^{(k)})']^{-1}\bbY^{(k)},
	\end{align*}
	and \begin{align*}
		&\bbX=\bbX^{(k)}+x_ke_k',\ \bbY=\bbY^{(k)}+y_ke_k', \ \bbX^{(k)}=\boldsymbol{\Lambda}\bbY^{(k)}+\boldsymbol{\Gamma}\bbW^{(k)}.
	\end{align*}
	By the calculation in Lemma \ref{lemma4}, we obtain that 
	\begin{align*}
		\bbH-\bbH^{(k)}=\bbP_{x}-\bbP_{x}^{(k)}+\bbP_{y}-\bbP_{y}^{(k)}=\mbox{the sum of 8 rank 1 matrices}:=\boldsymbol{\alpha}_k\boldsymbol{\beta}_k.
	\end{align*}
	Let $(\boldsymbol{\alpha}_k)_{n\times 8}$, $(\boldsymbol{\beta}_k)_{8\times n}$ denote the sets of column, row vector respectively, see Lemma \ref{lemma4} for more details. And by the formula (0.7.4.1) of \cite{HJ2012M}, we have 
	\begin{align*}
		(\bbH\!-\!z)^{-1}\!-\!(\bbH^{(k)}\!-\!z)^{-1}=\!-\!(\bbH^{(k)}\!-\!z)^{-1}\boldsymbol{\alpha}_k(\bbI_{8}\!+\!\boldsymbol{\beta}_k(\bbH^{(k)}\!-\!z)^{-1}\boldsymbol{\alpha}_k)^{-1}\boldsymbol{\beta}_k(\bbH^{(k)}\!-\!z)^{-1}.
	\end{align*}
	There exists a deterministic matrix $\bbM$ w.r.t. random vectors $y_k,w_k$, which approximates $\bbI_{8}+\boldsymbol{\beta}_k(\bbH^{(k)}-z)^{-1}\boldsymbol{\alpha}_k$ in the expectation sense by Lemma \ref{lemma5}.
	For this matrix with fixed dimension, we have
	\begin{align*}
		&\bbI_{8}\!+\!\boldsymbol{\beta}_k(\bbH^{(k)}\!-\!z)^{-1}\boldsymbol{\alpha}_k\!=\!\bbM\!+\!\tilde{\bbM},\quad\mbox{and}\quad (\bbI_{8}\!+\!\boldsymbol{\beta}_k(\bbH^{(k)}\!-\!z)^{-1}\boldsymbol{\alpha}_k)^{-1}=\bbM^{-1}\!+\!\widetilde{\bbM^{-1}}
	\end{align*}
	where $\mathbb{E}|\tilde{\bbM}_{ij}|^8=O(\frac{1}{n})$ and $\mathbb{E}|(\widetilde{\bbM^{-1}})_{ij}|^4=O(\frac{1}{n})$.
	Then martingale decomposition term becomes 
	\begin{align*}
		&\tr(\bbH\!-\!z)^{-1}\!-\!\tr(\bbH^{(k)}\!-\!z)^{-1}\\
		=&\!-\!\frac{1}{n}\sum_{i,j=1}^{8}\bbM_{ij}^{-1}\boldsymbol{\beta}_{k,j}(\bbH^{(k)}\!-\!z)^{-2}\boldsymbol{\alpha}_{k,i}\!+\!\frac{1}{n}\sum_{i,j=1}^{8}\widetilde{\bbM^{-1}}_{ij}\boldsymbol{\beta}_{k,j}(\bbH^{(k)}\!-\!z)^{-2}\boldsymbol{\alpha}_{k,i}.
	\end{align*}
	
First we consider the remainder term
\begin{align*}
	&\sum_{n=1}^{\infty}\mathbb{P}\left(\left|\sum_{k=1}^n(\mathbb{E}_k-\mathbb{E}_{k-1})\frac{1}{n}\sum_{i,j=1}^{8}\widetilde{\bbM^{-1}}_{ij}\boldsymbol{\beta}_{k,j}(\bbH^{(k)}-z)^{-2}\boldsymbol{\alpha}_{k,i}\right|>\varepsilon\right)\\
	\leq&\sum_{n=1}^{\infty}\frac{1}{\varepsilon^2}\mathbb{E}\left|\sum_{k=1}^n(\mathbb{E}_k-\mathbb{E}_{k-1})\frac{1}{n}\sum_{i,j=1}^{8}\widetilde{\bbM^{-1}}_{ij}\boldsymbol{\beta}_{k,j}(\bbH^{(k)}-z)^{-2}\boldsymbol{\alpha}_{k,i}\right|^2\\
	\leq&\sum_{n=1}^{\infty}\frac{1}{\varepsilon^2}\mathbb{E}\sum_{k=1}^n\left|(\mathbb{E}_k-\mathbb{E}_{k-1})\frac{1}{n}\sum_{i,j=1}^{8}\widetilde{\bbM^{-1}}_{ij}\boldsymbol{\beta}_{k,j}(\bbH^{(k)}-z)^{-2}\boldsymbol{\alpha}_{k,i}\right|^2\\
	\leq&\sum_{n=1}^{\infty}\frac{1}{\varepsilon^2}\mathbb{E}\sum_{k=1}^n\frac{1}{n^2}\left|\sum_{i,j=1}^{8}\widetilde{\bbM^{-1}}_{ij}\boldsymbol{\beta}_{k,j}(\bbH^{(k)}-z)^{-2}\boldsymbol{\alpha}_{k,i}\right|^2\\
	\leq&\sum_{n=1}^{\infty}\frac{1}{\varepsilon^2}\mathbb{E}\sum_{k=1}^n\frac{1}{n^2}\sum_{i,j,i',j'=1}^{8}\left|\widetilde{\bbM^{-1}}_{ij}\widetilde{\bbM^{-1}}_{i'j'}\right|^2n^{2\delta}+\sum_{n=1}^{\infty}\frac{1}{n}\mathbb{P}(\mathcal{A}^c)\\
	\leq&\sum_{n=1}^{\infty}\frac{1}{\varepsilon^2}\sum_{k=1}^n\frac{n^{2\delta}}{n^2}\sum_{i,j,i',j'=1}^{8}\sqrt{\mathbb{E}|\widetilde{\bbM^{-1}}_{ij}|^4}\sqrt{\mathbb{E}|\widetilde{\bbM^{-1}}_{i'j'}|^4}+\sum_{n=1}^{\infty}\frac{1}{n}\mathbb{P}(\mathcal{A}^c)\\
	=&\sum_{n=1}^{\infty}O\left(\frac{1}{n^{2(1-\delta)}\varepsilon^2}\right)+\sum_{n=1}^{\infty}O\left(\frac{1}{n^{1+2\delta}}\right)<\infty,
\end{align*}
where the first two inequalities are due to Chebyshev's inequality and Burkholder inequality, respectively, and we define the event
 \begin{align*}
 \mathcal{A}\!=\!\{\mbox{for any }i,j,i',j', |\boldsymbol{\beta}_{k,j}(\bbH^{(k)}\!-\!zI)^{-2}\boldsymbol{\alpha}_{k,i}|<n^{\delta}\ \mbox{and } |\boldsymbol{\beta}_{k,j'}(\bbH^{(k)}\!-\!z)^{-2}\boldsymbol{\alpha}_{k,i'}|<n^{\delta} \}.
 \end{align*}
 The specific forms of $\boldsymbol{\beta}_{k,j}(\bbH^{(k)}-z)^{-2}\boldsymbol{\alpha}_{k,i}$ include 
\begin{align*}
	\chi'_kG^2\chi_k,\ \omega'_kG^2 \omega_k,\ \chi'_kG^2\omega_k,\ e'_kG^2 \chi_k ,\  e'_kG^2 \omega_k,\ e'_kG^2e_k.
\end{align*}
By Cauchy-Schwarz inequality and the formula (2.3) of \cite{BaiS04C}, we have 
\begin{align*}
\mathbb{P}(\mathcal{A}^c)\leq \frac{8^4}{n^{2\delta}}\mathbb{E}\left|\boldsymbol{\beta}_{k,j}(\bbH^{(k)}-z)^{-2}\boldsymbol{\alpha}_{k,i}\boldsymbol{\beta}_{k,j'}(\bbH^{(k)}-z)^{-2}\boldsymbol{\alpha}_{k,i'}\right|^2=O\left(\frac{1}{n^{2\delta}}\right).
\end{align*} 
The remainder term can be ignorable in the sense of almost sure. Now we consider the main term
	\begin{align*}
		&\delta_k=\frac{1}{n}(\mathbb{E}_k-\mathbb{E}_{k-1})\left[\sum_{i,j=1}^{8}\bbM_{ij}^{-1}\boldsymbol{\beta}_{k,j}(\bbH^{(k)}-z)^{-2}\boldsymbol{\alpha}_{k,i}\right]\\
		=&\frac{1}{n}(\mathbb{E}_k-\mathbb{E}_{k-1})\left[\sum_{i,j=1}^{8}\bbM_{ij}^{-1}\left(\boldsymbol{\beta}_{k,j}(\bbH^{(k)}-z)^{-2}\boldsymbol{\alpha}_{k,i}-\tr(\bbH^{(k)}-z)^{-2}\mathbb{E}\boldsymbol{\alpha}_{k,i}\boldsymbol{\beta}_{k,j}\right)\right]
	\end{align*}
	\begin{align*}
		\sum_{n=1}^{\infty}\mathbb{P}\left(\left|\sum_{k=1}^n\delta_k\right|>\varepsilon\right)\leq\sum_{n=1}^{\infty}\frac{1}{\varepsilon^2}\mathbb{E}\left|\sum_{k=1}^n\delta_k\right|^2\leq\sum_{n=1}^{\infty}\frac{1}{\varepsilon^2}\sum_{k=1}^n\mathbb{E}\left|\delta_k\right|^2=\sum_{n=1}^{\infty}O\left(\frac{1}{n^2\varepsilon^2}\right)<\infty.
	\end{align*}
	The first two inequality are due to Markov inequality and Burkholder inequality, respectively, and the next equality is due to the fact $\mathbb{E}|\delta_k|^2=O(\frac{1}{n^3})$. By Borel-Cantelli Lemma, we have  $\tr(\bbH-z)^{-1}-\mathbb{E}\tr(\bbH-z)^{-1}\overset{a.s.}{\to}0.$

	\section{The equality for Stieltjes transform of LSD}
	From the arguments in Section \ref{s3} and \ref{s4}, we have
	\begin{align*}
		&\tr(\bbH-zI)^{-1}-\tr(\bbH_N-zI)^{-1}\\
		=&\tr(\bbH-zI)^{-1}-\mathbb{E}\tr(\bbH-zI)^{-1}+\mathbb{E}\tr(\bbH-zI)^{-1}-\mathbb{E}\tr(\bbH_N-zI)^{-1}\\
		&+\mathbb{E}\tr(\bbH_N-zI)^{-1}-\tr(\bbH_N-zI)^{-1}\overset{a.s.}{\to}0.
	\end{align*}
	From Theorem \ref{th11}, we have for the random variables with general underlying distribution, 
	\begin{align*}
		\frac{1}{p+q}\Tr G_b(z)=&(z-1)(\frac{1}{p+q}\Tr G_{xy}((1-z)^2)+\frac{1}{p+q}\Tr G_{yx}((1-z)^2))\\
		=&(z-1)\left(\frac{q-p}{p+q}\frac{1}{-(1-z)^2}+\frac{2p}{p+q}\frac{1}{p}\Tr G_{xy}((1-z)^2)\right),
	\end{align*}
	where we use the fact that the spectrum of $(\bbX\bbX')^{-\frac{1}{2}}\bbX\bbP_y\bbX'(\bbX\bbX')^{-\frac{1}{2}}$ and $(\bbY\bbY')^{-\frac{1}{2}}\bbY\bbP_x\bbY'(\bbY\bbY')^{-\frac{1}{2}}$ differs by $|p-q|$ zero eigenvalues. 
	That is $f(\frac{1}{p}\Tr G_{xy}((1-z)^2))-f(\frac{1}{p}\Tr G_{xy}^{N}((1-z)^2))\overset{a.s.}{\to}0$, then we have $\frac{1}{p}\Tr G_{xy}((1-z)^2)-\frac{1}{p}\Tr G_{xy}^N((1-z)^2)\overset{a.s.}{\to}0$.
	
	From the relationship in \cite{BHHJZ2020}, under Gaussian assumption the Stieltjes transform of the SCC matrix and the corresponding noncentral Fisher matrix satisfies
	\begin{align}\label{mandmf}
		m(z)&=\int\frac{1}{t-z}dF^{SCC}(t)=\int\frac{1}{t-z}dF^{F}\left(\frac{(n-q)t}{q(1-t)}\right)\\
		&=\int\frac{1}{\frac{q}{n-q}x/(1+\frac{q}{n-q}x)-z}dF^{F}(x)\nonumber\\
		&=\int\frac{1+\frac{q}{n-q}x}{\frac{q}{n-q}x-z(1+\frac{q}{n-q}x)}dF^{F}(x)=\frac{n-q}{q(1-z)}\int\frac{1+\frac{q}{n-q}x}{x-\frac{z(n-q)}{q(1-z)}}dF^F(x)\nonumber\\
		&=\frac{n-q}{q(1-z)}\int\frac{\frac{q}{n-q}[x-\frac{z(n-q)}{q(1-z)}]+\frac{z}{1-z}+1}{x-\frac{z(n-q)}{q(1-z)}}dF^F(x)\nonumber\\
		&=\frac{1}{1\!-\!z}+\frac{n\!-\!q}{q(1\!-\!z)^2}\int\frac{1}{x\!-\!\frac{z(n\!-\!q)}{q(1\!-\!z)}}dF^F(x)=\frac{1}{1\!-\!z}+\frac{n\!-\!q}{q(1\!-\!z)^2}m_F\left(\frac{z(n\!-\!q)}{q(1\!-\!z)}\right). \nonumber
	\end{align}
	The second equality in the first line is due to the relationship between $F^{SCC}(\cdot)$ and $F^{F}(\cdot)$,
         \begin{align*}
         F^{SCC}(t)\!=\!\frac{1}{p}\sum_{i=1}^pI\{\lambda_i^{SCC}\leq t\}\!=\!\frac{1}{p}\sum_{i=1}^pI\left\{\frac{(n\!-\!q)\lambda_i^{SCC}}{q(1\!-\!\lambda_i^{SCC})}\leq \frac{(n\!-\!q)t}{q(1\!-\!t)}\right\}\!=\!F^{F}\left(\frac{(n\!-\!q)t}{q(1\!-\!t)}\right).
         \end{align*}
	The equality in the second line is due to the variable substitution 
	\begin{align*}
		x=\frac{(n-q)t}{q(1-t)}\Leftrightarrow t=\frac{qx}{n-q}/(1+\frac{q}{n-q}x).
	\end{align*}
	
	Notice $F^F(\cdot)$ is the LSD of noncentral Fisher matrix with noncentral parameter matrix $\boldsymbol{\Xi}=\frac{n}{q}\bbT(\frac{1}{n}\bbQ\bbQ)^{-1}\bbT'$, where the eigenvalues of $\bbT\bbT'$ is $\frac{(\lambda_i^{\boldsymbol{\Lambda}})^2}{1-(\lambda_i^{\boldsymbol{\Lambda}})^2}$, and the dimension of $\bbQ$ is $p\times n$. By the results of sample covariance matrix, we have the Stieltjes transform of the LSD of $\boldsymbol{\Xi}$ satisfies
	\begin{align*}
		m_{\boldsymbol{\Xi}}(z)=\int\frac{1}{t(1-\frac{p}{n}-\frac{p}{n}zm_{\boldsymbol{\Xi}}(z))-z}dH\left(\sqrt{\frac{\frac{q}{n}t}{1+\frac{q}{n}t}}\right)
	\end{align*}
	where we use the fact 
	\begin{align*}
		H(x)\!=\!\lim_{p\to\infty}\frac{1}{p}\sum_{i=1}^pI\{\lambda_i^{\boldsymbol{\Lambda}}\leq x\}\!=\!\lim_{p\to\infty}\frac{1}{p}\sum_{i=1}^pI\left\{\frac{n/q(\lambda_i^{\boldsymbol{\Lambda}})^2}{1\!-\!(\lambda_i^{\boldsymbol{\Lambda}})^2}\leq \frac{n/qx^2}{1\!-\!x^2}\right\}\!=\!H^{\frac{n}{q}TT'}\left(\frac{n/qx^2}{1\!-\!x^2}\right),
	\end{align*}
	and by variable substitution, we have 
	\begin{align*}
		H^{\frac{n}{q}TT'}(t)=H\left(\sqrt{\frac{\frac{q}{n}t}{1+\frac{q}{n}t}}\right).
	\end{align*}
	According to the main result in \cite{Zhang2023L}, we have the Stieltjes transform of $F^F(\cdot)$ with noncentral parameter matrix $\boldsymbol{\Xi}$ satisfies
	\begin{align*}
		m_F(z)\!=\!&\int\left[\frac{t}{1\!+\!(\frac{p}{q}\!+\!\frac{p}{n\!-\!q}z)m_F}\!+\!\frac{1\!-\!\frac{p}{q}}{1\!+\!\frac{p}{n\!-\!q}zm_F}\!-\!z\frac{1\!+\!(\frac{p}{q}\!+\!\frac{p}{n\!-\!q}z)m_F}{1\!+\!\frac{p}{n\!-\!q}zm_F}\right]^{-1}dH^{\boldsymbol{\Xi}}(t)\\
		\frac{m_F(z)}{1\!+\!(\frac{p}{q}\!+\!\frac{p}{n-q}z)m_F}\!=\!&\int\left[t\!+\!\frac{(1\!-\!\frac{p}{q})[1\!+\!(\frac{p}{q}\!+\!\frac{p}{n\!-\!q}z)m_F]}{1\!+\!\frac{p}{n\!-\!q}zm_F}\!-\!z\frac{[1\!+\!(\frac{p}{q}\!+\!\frac{p}{n\!-\!q}z)m_F]^2}{1\!+\!\frac{p}{n\!-\!q}zm_F}\right]^{-1}dH^{\boldsymbol{\Xi}}(t)\\
		=&m_{\boldsymbol{\Xi}}\left(z\frac{[1+(\frac{p}{q}+\frac{p}{n-q}z)m_F]^2}{1+\frac{p}{n-q}zm_F}-\frac{(1-\frac{p}{q})[1+(\frac{p}{q}+\frac{p}{n-q}z)m_F]}{1+\frac{p}{n-q}zm_F}\right).
	\end{align*}
	That means the pair 
	\begin{align*}
		\left(z\frac{[1+(\frac{p}{q}+\frac{p}{n-q}z)m_F]^2}{1+\frac{p}{n-q}zm_F}-\frac{(1-\frac{p}{q})[1+(\frac{p}{q}+\frac{p}{n-q}z)m_F]}{1+\frac{p}{n-q}zm_F},\quad \frac{m_F(z)}{1+(\frac{p}{q}+\frac{p}{n-q}z)m_F}\right)
	\end{align*}
	satisfies the equation of the pair $(a,b)$
	\begin{align}\label{equationab}
		b=\int\frac{1}{t(1-\frac{p}{n}-\frac{p}{n}ab)-a}dH\left(\sqrt{\frac{\frac{q}{n}t}{1+\frac{q}{n}t}}\right).
	\end{align}
	According to the relationship \eqref{mandmf} between $m$ and $m_F$, we have the pair 
	\begin{align*}
		\frac{1\!+\!\frac{p}{n\!-\!q}[(1\!-\!z)m\!-\!1]}{1\!+\!\frac{pz}{n\!-\!q}[(1\!-\!z)m\!-\!1]}\left\{\frac{z(n\!-\!q)}{q(1\!-\!z)}\left[1\!+\!\frac{p[(1\!-\!z)m\!-\!1]}{n\!-\!q}\right]\!-\!\left(1\!-\!\frac{p}{q}\right)\right\}
	\end{align*}
	and
	\begin{align*}
		\frac{\frac{q(1\!-\!z)}{n\!-\!q}[(1\!-\!z)m\!-\!1]}{1\!+\!\frac{p}{n\!-\!q}[(1\!-\!z)m\!-\!1]}
	\end{align*}
	satisfies the same equation of the pair $(a,b)$.

	Then substituting the above pair in formula \eqref{equationab}, we have 
	\begin{align*}
		\frac{\frac{q(1-z)}{n-q}[(1-z)m-1]}{1+\frac{p}{n-q}[(1-z)m-1]}=\int\frac{1}{tG_1(z,m)-G_2(z,m)}dH\left(\sqrt{\frac{\frac{q}{n}t}{1\!+\!\frac{q}{n}t}}\right)
	\end{align*}
	\begin{align}\label{CCALSD}
		\Leftrightarrow G_1(z,m)\frac{\frac{q(1-z)}{n-q}[(1-z)m-1]}{1+\frac{p}{n-q}[(1-z)m-1]}=\int_0^1\left[\frac{n}{q}\frac{x^2}{1-x^2}-\frac{G_2(z,m)}{G_1(z,m)}\right]^{-1}dH(x)
	\end{align}
	where 
	\begin{align*}
		G_1(z,m)\!=&1\!-\!\frac{p}{n}\!-\!\frac{p}{n}\frac{\frac{q(1\!-\!z)}{n\!-\!q}[(1\!-\!z)m\!-\!1]}{1\!+\!\frac{pz}{n\!-\!q}[(1\!-\!z)m\!-\!1]}\left\{\frac{z(n\!-\!q)}{q(1\!-\!z)}\left[1\!+\!\frac{p[(1\!-\!z)m\!-\!1]}{n\!-\!q}\right]-\left(1\!-\!\frac{p}{q}\right)\right\}\\
		G_2(z,m)\!=&\frac{1\!+\!\frac{p}{n\!-\!q}[(1\!-\!z)m\!-\!1]}{1\!+\!\frac{pz}{n\!-\!q}[(1\!-\!\!z)m\!-\!1]}\left\{\frac{z(n\!-\!q)}{q(1\!-\!z)}\left[1\!+\!\frac{p}{n\!-\!q}[(1\!-\!z)m\!-\!1]\right]\!-\!\left(1\!-\!\frac{p}{q}\right)\right\}.
	\end{align*}
	The Stieltjes transform of the LSD of the SCC matrix satisfies the equation \eqref{CCALSD}.

	According to the relationship between $\Tr G_b$ and $\Tr G_{xy}$ in Theorem \ref{th11}, we have the pair 
	\begin{align*}
     \left((1-z)^2,\frac{(c_1+c_2)m_b+\frac{c_1-c_2}{1-z}}{2c_1(z-1)}\right)
	\end{align*} satisfies the equation \eqref{CCALSD} as $(z,m)$, where $m_b$ means the limit of $\frac{1}{p+q}\Tr G_b$, the Stieltjes transform of the LSD of the block correlation matrix $\bbB$.
	
	Moreover, we choose Gamma distribution samples with dimension setting: $(p,q,n)=(1000,3000,8000)$, to give a histogram for comparison between the ESD and the LSD of $\mathcal{S}_{xy}$. In first subfigure of Figure \ref{fig}, we consider the fixed rank alternative case, that is the same with null case except some outliers. In second subfigure of Figure \ref{fig}, we consider the case of rank $p/2$, and there is an interesting phenomenon because of the spectrum gap of $H$, where more details and further results need more exploration of the properties of LSD. In last subfigure of Figure \ref{fig}, we consider the full rank alternative case, and they all fit well with our theoretical results.
\begin{figure}[h]
	\begin{minipage}{0.32\linewidth}
		\vspace{3pt}
		\centerline{\includegraphics[width=\textwidth]{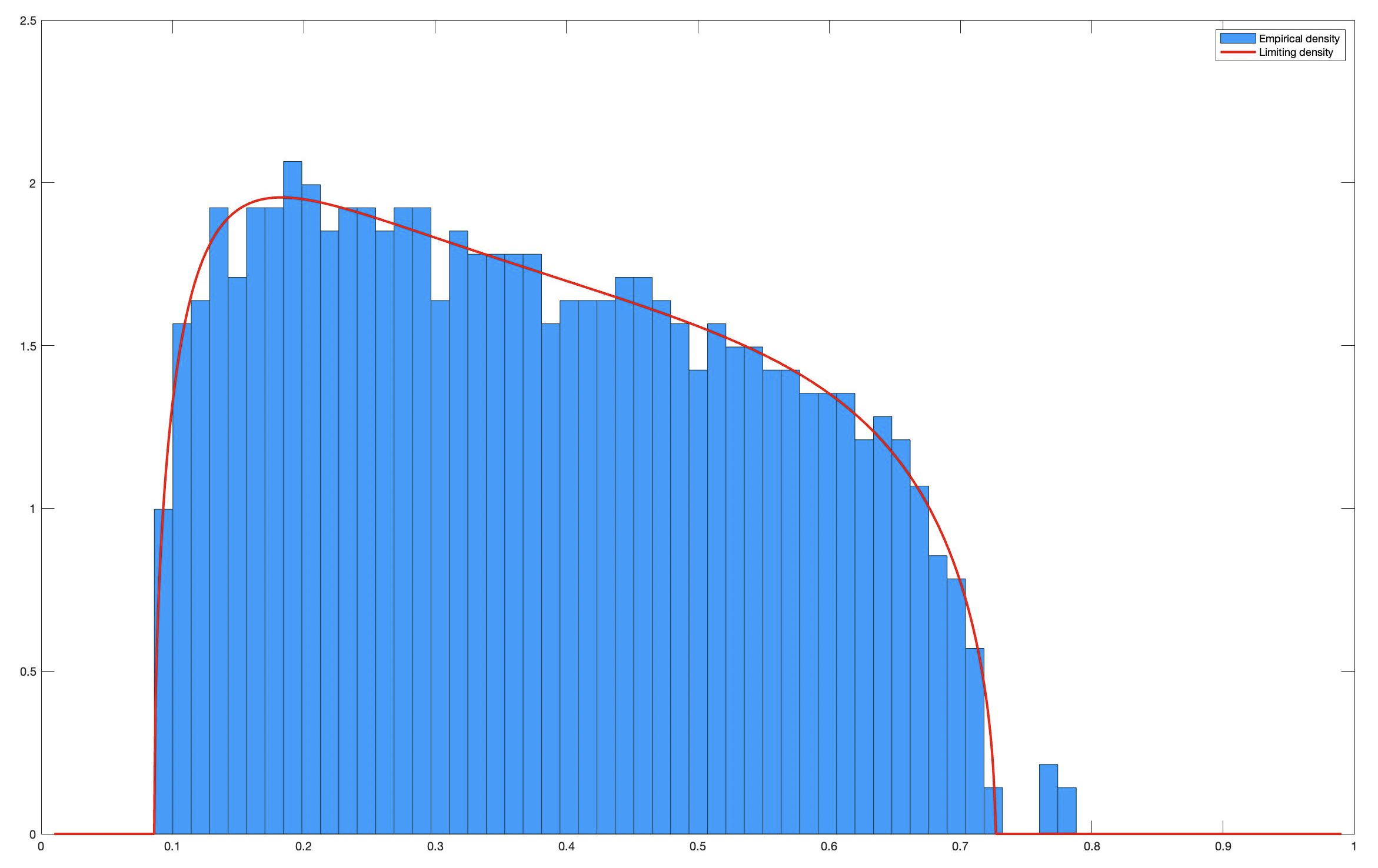}}
		\centerline{$\boldsymbol{\Lambda}\boldsymbol{\Lambda}'=\frac{1}{2}\bbI_5\oplus 0\bbI_{p-5}$}
	\end{minipage}
	\begin{minipage}{0.32\linewidth}
		\vspace{3pt}
		\centerline{\includegraphics[width=\textwidth]{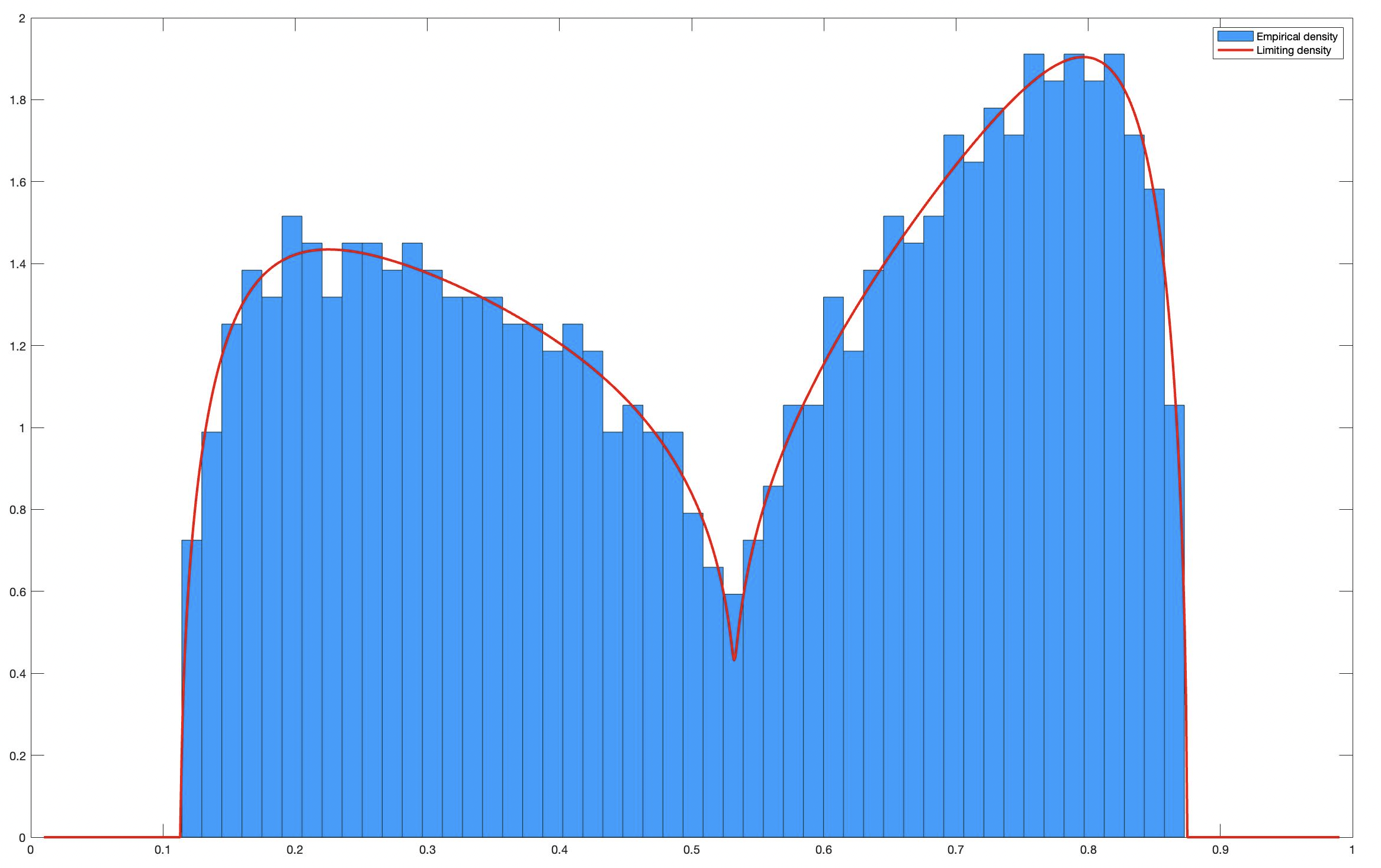}}
		\centerline{$\boldsymbol{\Lambda}\boldsymbol{\Lambda}'=\frac{1}{2}\bbI_{p/2}\oplus 0\bbI_{p/2}$}
	\end{minipage}
	\begin{minipage}{0.32\linewidth}
		\vspace{3pt}
		\centerline{\includegraphics[width=\textwidth]{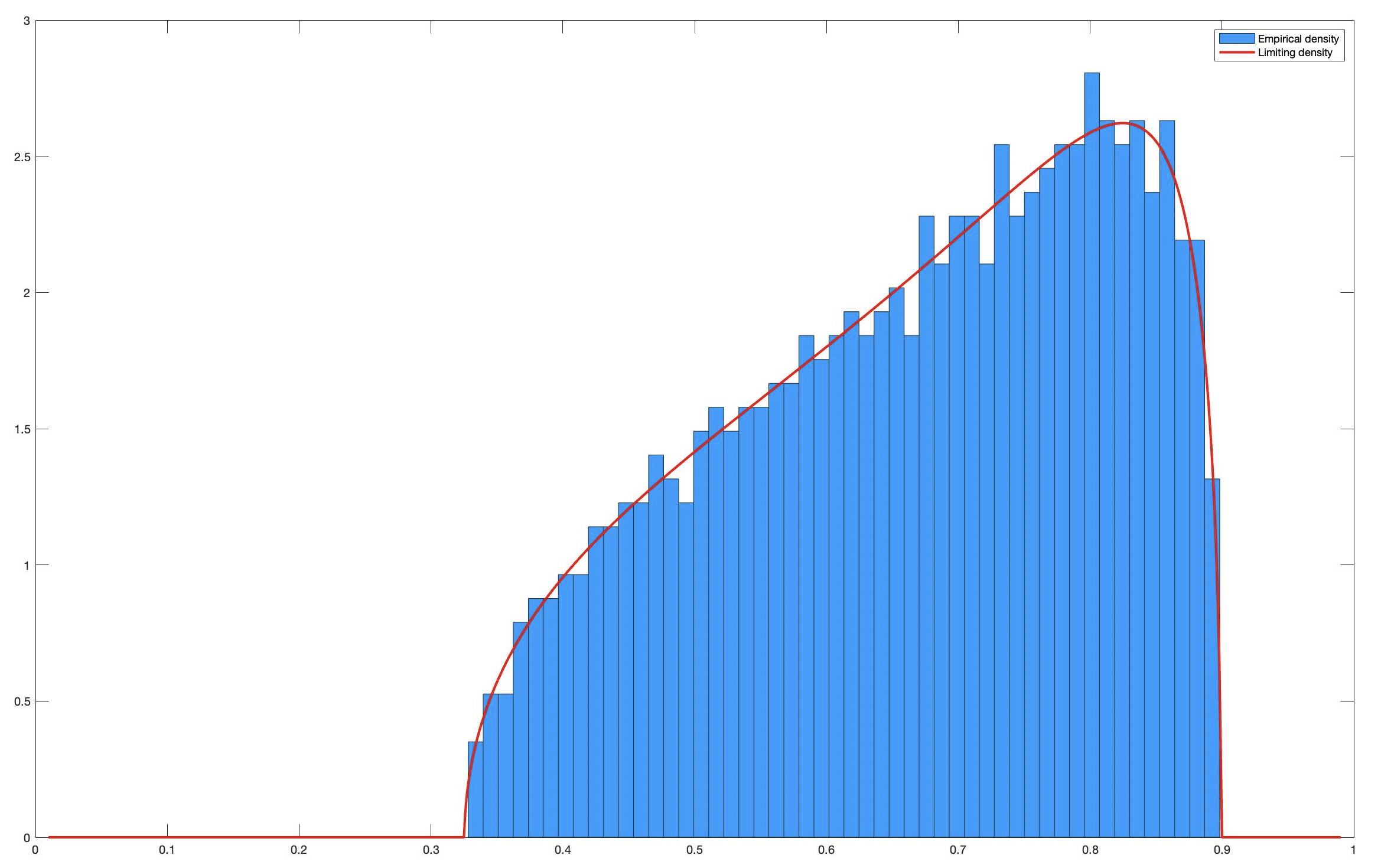}}
		\centerline{$\boldsymbol{\Lambda}\boldsymbol{\Lambda}'=\frac{1}{2}\bbI_p$}
	\end{minipage}
	\caption{LSD of SCC matrix with underlying distribution G(4,2).}
	\label{fig}
\end{figure}

\section{Some Lemmas}

\begin{lemma}\label{lemma0}
Consider the real matrix $\bbB$ and $\bbA$ have the following structure:
\begin{align*}
\bbB=\begin{pmatrix} \bbI_p, &\bbA\\
\bbA', &\bbI_q \end{pmatrix},
\end{align*}
and denote the $\min\{p+q\}$ largest singular values of $\bbA$ by $l_1\geq l_2\geq \dots \geq l_{\min\{p,q\}}$, then the $p+q$ eigenvalues of $\bbB$ will consist of 
\begin{align*}
\{\underbrace{1+l_1,\dots,1+l_{\min\{p,q\}}}_{\min\{p,q\}},\underbrace{1,\dots,1}_{\max\{p,q\}-\min\{p,q\}},\underbrace{1-l_{\min\{p,q\}},\dots,1-l_1}_{\min\{p,q\}}\}.
\end{align*} 
\end{lemma}

\begin{proof}
Without loss of generality, we assume $p<q$. Set singular value decomposition (SVD) for $\bbA_{p\times q}=\bbU \bbL\bbV'$, where $\bbU_{p\times p}=(u_1,\dots,u_p)$, $\bbV_{q\times q}=(v_1,\dots,v_q)$, $\bbU'\bbU=\bbI_p$, $\bbV'\bbV=\bbI_q$. 
Denote the eigenvalue of $\bbB$ by $\lambda$.

First, we will show $\{1+l_1,\dots, 1+l_i, \dots,1+l_p\}$ are the eigenvalues of $\bbB$ with eigenvectors $\{(u_1',v_1')',\dots,(u_i',v_i')',\dots,(u_p',v_p')'\}$.  According to SVD for $\bbA$, we have $\bbA v_i=\sum_{j=1}^p l_ju_jv_j'v_i=l_iu_i$ and $\bbA' u_i=\sum_{j=1}^p l_jv_ju_j'u_i=l_iv_i$,  and then 
\begin{align*}
\begin{pmatrix} \bbI_p, & \bbA\\    \bbA', &\bbI_q \end{pmatrix}
\begin{pmatrix} u_i \\ v_i  \end{pmatrix}=
\begin{pmatrix} \bbI_p u_i+\bbA v_i \\ \bbA'u_i +\bbI_q v_i\end{pmatrix}=
\begin{pmatrix} (1+l_i) u_i \\ (1+l_i)v_i\end{pmatrix}=
(1+l_i)\begin{pmatrix} u_i \\ v_i  \end{pmatrix}.
\end{align*}

Second, we will show $\{1-l_1,\dots, 1-l_i, \dots,1-l_p\}$ are the eigenvalues of $\bbB$ with eigenvectors $\{(\tilde{u}_{p+1}',\tilde{v}_{p+1}')',\dots,(\tilde{u}_{p+i}',\tilde{v}_{p+i}')',\dots,(\tilde{u}_{2p}',\tilde{v}_{2p}')'\}$. According to the definition of eigenvalue and eigenvector, we have 
\begin{align*}
\begin{pmatrix} \bbI_p, & \bbA\\    \bbA', &\bbI_q \end{pmatrix}
\begin{pmatrix} u \\ v  \end{pmatrix}=\begin{pmatrix}  u+\bbA v \\ \bbA'u+v\end{pmatrix}=
\lambda \begin{pmatrix} u\\ v  \end{pmatrix}.
\end{align*}
Because $(u',v')'$ and $(u_{i}',v_{i}')'$ are the eigenvectors of $\bbB$ with different eigenvalues $\lambda$ and $1+l_i$, they are orthogonal, i.e., $u_i'u+v_i'v=0$ for any $i=1, \dots ,p$. The above equals to 
\begin{equation}\label{eq31}
\begin{pmatrix} (1-\lambda)u+\bbA v \\  \bbA'u+(1-\lambda)v\end{pmatrix}=0\ \Rightarrow 
\left\{\begin{array}{lr} (1-\lambda)u_i'u+u_i'\bbA v=0\\
v_i'\bbA' u+(1-\lambda)v_i' v=0 \end{array}. \right.
\end{equation}
We have 
\begin{align*}
u_i'\bbA=(\bbA' u_i)'=(l_iv_i)'=l_iv_i',\ \ v_i'\bbA' =(\bbA v_i)'=(l_iu_i)'=l_iu_i'.
\end{align*}
Then the equation set \eqref{eq31} becames
\begin{equation*}
\left\{ {\begin{array}{l}
{(1-\lambda)[u_i'u+v_i'v]+[l_i-(1-\lambda)] v_i'v=0}\\
{(1-\lambda)[u_i'u+v_i'v]+[l_i-(1-\lambda)] u_i' u=0}
 \end{array}} \right.
\Rightarrow \left\{ {
\begin{array}{l}
{[l_i-(1-\lambda)] v_i'v=0} \\
{[l_i-(1-\lambda)] u_i'u=0 }
\end{array}} \right. .
\end{equation*}
We confirm $\lambda=1-l_i$.  If not, the dimension of the space spanned by $(u_1,u_2,\dots,u_p)$ equals to $p$, there is no $u$ that satisfies $u_i'u=0$ for all $i$.  When the index $i$ varies, the corresponding eigenvector $(u,v)$ changes accordingly, and will fill the orthogonal subspace of the subspace spanned by $\{(u_1',v_1')',\dots,(u_i',v_i')',$ $\dots, (u_p',v_p')'\}$.

Last, we will show the remaining eigenvalues of $\bbB$ are all $1$. There are $q-p$ $v$ such that each is orthogonal to $\{v_1,\dots,v_p\}$, because there are $q$ orthogonal columns in $\bbV$.  Extend the base $\{v_{p+1}, v_{p+2},\dots,v_{q}\}$ in $q$-dimensional linear space to the base $\{(0',v_{p+1}')',(0',v_{p+2}')', \dots,(0',v_q')'\}$ in $p+q$-dimensional linear space. 
Now, $\bbA v_{p+i}=\sum_{j=1}^pl_ju_jv_j'v_{p+i}=0$  and $\bbA'u_{p+i}=0$. That is,
\begin{align*}
\begin{pmatrix} \bbI_p, & \bbA\\    \bbA', &\bbI_q \end{pmatrix}
\begin{pmatrix} 0 \\ v_{p+i}  \end{pmatrix}=\begin{pmatrix}  \bbA v_{p+i} \\ v_{p+i}\end{pmatrix}=
1\begin{pmatrix} 0 \\ v_{p+i}  \end{pmatrix}.
\end{align*}
\end{proof}

\begin{lemma}\label{lemma1}
Under the assumptions and notations of Theorem \ref{th12}, we have
\begin{align*}
\frac{1}{\sqrt{1-s}}\frac{\partial \Phi_{ji}}{\partial \hat{W}_{ji}}=\frac{1}{\sqrt{s}}\frac{\partial \Phi_{ji}}{\partial W_{ji}}.
\end{align*}
\end{lemma}
\begin{proof} According to the notation \eqref{phi} and chain rule, we have 
\begin{align*}
	&\frac{\partial \Phi_{ji}}{\partial \hat{W}_{ji}}=\frac{\partial[\boldsymbol{\Gamma}'\bbQ(s)(\bbH(s)-zI)^{-2}(\bbI-\bbP_{x(s)})]_{ji}}{\partial \hat{W}_{ji}},\\
	=&\left[\boldsymbol{\Gamma}'\frac{\partial\bbQ(s)}{\partial \hat{W}_{ji}}G^2(s)(\bbI-\bbP_{x(s)})-\boldsymbol{\Gamma}'\bbQ(s)G(s)\frac{\partial \bbP_{x(s)}}{\partial \hat{W}_{ji}}G^2(s)(\bbI-\bbP_{x(s)})\right. \\
	&-\left. \boldsymbol{\Gamma}'\bbQ(s)G^2(s)\frac{\partial \bbP_{x(s)}}{\partial \hat{W}_{ji}}G(s)(\bbI-\bbP_{x(s)})-\boldsymbol{\Gamma}'\bbQ(s)G^2(s)\frac{\partial \bbP_{x(s)}}{\partial \hat{W}_{ji}}\right]_{ji}.
	\end{align*}
	To calculate the above derivatives, we need some matrix derivative results
	\begin{align*}
		\frac{\partial \bbX(s)}{\partial \hat{W}_{ji}}&=\sqrt{1\!-\!s}\boldsymbol{\Gamma}1_{ji},\  \frac{\partial \bbX(s)'}{\partial \hat{W}_{ji}}=\sqrt{1\!-\!s}1_{ij}\boldsymbol{\Gamma}',\  \frac{\partial \bbX(s)}{\partial {W}_{ji}}=\sqrt{s}\boldsymbol{\Gamma}1_{ji},\ \frac{\partial \bbX(s)'}{\partial {W}_{ji}}=\sqrt{s}1_{ij}\boldsymbol{\Gamma}',\\
		\frac{\partial \bbQ(s)}{\partial \hat{W}_{ji}}&=\!-\!(\bbX(s)\bbX(s)')^{-1}\left(\frac{\partial \bbX(s)}{\partial \hat{W}_{ji}}\bbX(s)'\!+\!\bbX(s)\frac{\partial \bbX(s)'}{\partial \hat{W}_{ji}}\right)\bbQ(s)\!+\!(\bbX(s)\bbX(s)')^{-1}\frac{\partial \bbX(s)}{\partial \hat{W}_{ji}}\\
		&=\sqrt{1-s}[(\bbX(s)\bbX(s)')^{-1}\boldsymbol{\Gamma}1_{ji}(\bbI-\bbP_{x(s)})-\bbQ(s)1_{ij}\boldsymbol{\Gamma}'\bbQ(s)],\\
		\frac{\partial \bbQ(s)}{\partial {W}_{ji}}&=\sqrt{s}[(\bbX(s)\bbX(s)')^{-1}\boldsymbol{\Gamma}1_{ji}(\bbI-\bbP_{x(s)})-\bbQ(s)1_{ij}\boldsymbol{\Gamma}'\bbQ(s)],
	\end{align*}
	\begin{align*}
		\frac{\partial\bbP_{x(s)}}{\partial \hat{W}_{ji}}=&\frac{\partial \bbX(s)'}{\partial \hat{W}_{ji}}\bbQ(s)+\bbX(s)'\frac{\partial \bbQ(s)}{\partial \hat{W}_{ji}}\\
		=&\sqrt{1-s}1_{ij}\boldsymbol{\Gamma}'\bbQ(s)+\sqrt{1-s}[\bbQ'(s)\boldsymbol{\Gamma}1_{ji}(\bbI-\bbP_{x(s)})-\bbP_{x(s)}1_{ij}\boldsymbol{\Gamma}'\bbQ(s)]\\
		=&\sqrt{1-s}(\bbI-\bbP_{x(s)})1_{ij}\boldsymbol{\Gamma}'\bbQ(s)+\sqrt{1-s}\bbQ'(s)\boldsymbol{\Gamma}1_{ji}(\bbI-\bbP_{x(s)}),\\
		\frac{\partial\bbP_{x(s)}}{\partial {W}_{ji}}=&\sqrt{s}(\bbI-\bbP_{x(s)})1_{ij}\boldsymbol{\Gamma}'\bbQ(s)+\sqrt{s}\bbQ'(s)\boldsymbol{\Gamma}1_{ji}(\bbI-\bbP_{x(s)}).
	\end{align*}
After substituting the above results into $\frac{\partial \Phi_{ji}}{\partial \hat{W}_{ji}}$ and $\frac{\partial \Phi_{ji}}{\partial W_{ji}}$ respectively, the proof is completed.
\end{proof}

	\begin{lemma}\label{lemma2}
	Under the assumptions of  Theorem \ref{th12}, for matrix $\bbQ(s)$, we have the operator norm of $\bbX(s)$ and $(\bbX(s)\bbX(s)')^{-1}$ is order $1$ with high probability, respectively, that is, there exist $b^{+}$ and $b^{-}$ such that
	\begin{align*}
	\mathbb{P}(s_1(\bbX(s))>2b^+)=o(n^{-l}), \  \mathbb{P}\left(\mbox{for all large n, } \|(\bbX(s)\bbX'(s))^{-1}\|\leq \frac{1}{kb^-}\right)=1,
	\end{align*}
	where $k$ is related to the operator norm of $\boldsymbol{\Lambda}$.
	\end{lemma}
	\begin{proof}
	According to the result of the largest eigenvalues of sample covariance matrix, we have the estimate \begin{align*}
			\mathbb{P}(s_1(\bbW(s))>2)=o(n^{-l})
		\end{align*}
		holds for any $l$. And it holds when replacing $\bbW(s)$ with $\bbY$. Even we do not need to truncate their entries.
		
		Obviously $\|\boldsymbol{\Gamma}\|<1$, define $\bbC=\boldsymbol{\Gamma}^{-1}\boldsymbol{\Lambda}$ and choose constant $b^+$ such that $s_1(\bbC)=b^+-2$. According to Theorem A.8 and A.9 in \cite{BS2010}, we have 
		\begin{align*}
			&s_1(\bbX(s))\leq \|\boldsymbol{\Gamma}\|s_1(\bbC\bbY+\bbW(s)))\leq s_1(\bbC\bbY)+s_1(\bbW(s)).
			\end{align*}
			We can easily get
			\begin{align*}
			&\mathbb{P}(s_1(\bbX(s))>2b^+)\leq	\mathbb{P}(s_1(\bbC)s_1(\bbY)+s_1(\bbW(s))>2b^+)\\
			\leq&\mathbb{P}(\{s_1(\bbY)>2\} \cup \{s_1(\bbW(s))>2\} \cup  \{s_1(\bbC)>b^+-1\})\\
			\leq&\mathbb{P}(s_1(\bbY)>2 )+\mathbb{P}(s_1(\bbW(s))>2 )+\mathbb{P}(s_1(\bbC)>b^+-1)=o(n^{-l}).
		\end{align*}

		Next we consider the operator norm of $(\bbX(s)\bbX(s)')^{-1}$, which equals to that of $[(\bbC\bbY+\bbW(s))(\bbC\bbY+\bbW(s))']^{-1}$. First we know when $c_1<1$, LSD of the later one has no mass at $0$. Next, we need to prove $0$ is outside the support, which is due to the sufficient condition for determining the support. Let $b$ satisfies the following equation set
\begin{equation*}
\left\{
\begin{aligned}
\frac{m}{1+c_1m}&=\int\frac{dH^{y}(t)}{t+(1-c_1)(1+c_1m)}\\
b&=1+c_1m
\end{aligned}
\right.
\end{equation*}
then we have 
\begin{align}\label{formula1}
&\frac{1}{c_1}\left(1-\frac{1}{b}\right)=m_{H^{y}}(-(1-c_1)(1+c_1m))=m_{H^{y}}\left(-(1-c_1)b\right).
\end{align}
When $b\to 0^{+}$, the left hand side $\frac{1}{c_1}\left(1-\frac{1}{b}\right)\to-\infty$ and the right hand side $m_{H^{y}}\left(-(1-c_1)b\right)\to\int_{0}^{+\infty}\frac{1}{t}dH^{y}(t)>0$; when $b\to +\infty$, the left hand side $\frac{1}{c_1}\left(1-\frac{1}{b}\right)\to\frac{1}{c_1}>0$, and the right hand side $m_{H^y}\left(-(1-c_1)b\right)\to0^{+}$. Noting the left hand side is monotone continuous function, according to Existence Theorem of Zero Points, we have the above equation about $b\in \mathbb{R}$ has at least one real solution in positive axis. Then we analysis the positive solution $b$.
By formula \eqref{formula1} and basic calculation
\begin{align*}
&m_{H^y}^{-1}\left(\frac{1}{c_1}\left(1-\frac{1}{b}\right)\right)=-(1-c_1)(1+c_1m)=-(1-c_1)b,
\end{align*}
 we have the derivatives of discriminant function
\begin{align*}
x'(b)=\frac{-2}{b^2}m_{H^y}^{-1}\left(\frac{1}{c_1}\left(1\!-\!\frac{1}{b}\right)\right)\!+\!\frac{1}{b^2}\frac{\partial m_{H^y}^{-1}(\frac{1}{c_!}(1\!-\!\frac{1}{b}))}{\partial \frac{1}{c_1}(1\!-\!\frac{1}{b})}\frac{\partial \frac{1}{c_1}(1\!-\!\frac{1}{b})}{\partial b}\!+\!\frac{-1}{b^2}(1\!-\!c_1),
\end{align*}
where
\begin{align*}
m_{H^y}(z)=\int\frac{dH^y(t)}{t-z},\ \ m'_{H^y}(z)=\int\frac{1}{(t-z)^2}dH^y(t),\ \ (m_{H^y}^{-1})'(z)=\frac{1}{m'_{H^y}(z)}.
\end{align*}
Moreover, we have 
\begin{align*}
x'(b)=\frac{1-c_1}{b^2}+\frac{1}{b^4c_1}\left[\int\frac{dH^y(t)}{[t-\frac{1}{c_1}(1-\frac{1}{b})]^2}\right]^{-1}>0,
\end{align*}
which implies
\begin{align*}
x(b)=\frac{1}{b^2}m_{H^y}^{-1}\left(\frac{1}{c_1}\left(1-\frac{1}{b}\right)\right)+\frac{1}{b}(1-c_1)=0\in S_F^c.
\end{align*}

Considering $S_F^c$ is open set, then $0$ belongs to it, there exists $b^->0$ such that $[0,b^{-})\subset S_F^c$. Combining the spectral property of $\boldsymbol{\Lambda}$ and the exact separation result for the information plus noise matrix, we have for some existing $b^->0$
	\begin{align*}
	&\mathbb{P}(\mbox{for all large n, } \lambda_{min}[(\bbC\bbY+\bbW(s))(\bbC\bbY+\bbW(s))']\geq b^-)\\
	\leq&\mathbb{P}\left(\mbox{for all large n, } \|(\bbX(s)\bbX'(s))^{-1}\|\leq \frac{1}{kb^-}\right)=1.
	\end{align*}
	\end{proof}

	\begin{lemma}\label{lemma3}
	Similar with Lemma \ref{lemma1}, under the assumptions and notations of Theorem \ref{th12}, we have for $l=1,2$,
\begin{align}\label{formula2}
\frac{1}{\sqrt{1-t}}\frac{\partial \Psi_{ji}^l}{\partial \hat{Y}_{ji}}=\frac{1}{\sqrt{t}}\frac{\partial \Psi_{ji}^l}{\partial Y_{ji}}.
\end{align}
	\end{lemma}
\begin{proof}
		For $l=1$, we can easily get
	\begin{align*}
		&\frac{\partial \Psi^1_{ji}}{\partial \hat{Y}_{ji}}=\frac{\partial[\boldsymbol{\Lambda}'\bbQ(t)(\bbH(t)-zI)^{-2}(\bbI-\bbP_{x(t)})]_{ji}}{\partial \hat{Y}_{ji}}\\
		=&\left[\boldsymbol{\Lambda}'\frac{\partial\bbQ(t)}{\partial \hat{Y}_{ji}}G^2(t)(\bbI-\bbP_{x(t)})\right]_{ji}\!-\!\left[\boldsymbol{\Lambda}'\bbQ(t)G(t)\left(\frac{\partial \bbP_{x(t)}}{\partial \hat{Y}_{ji}}\!+\!\frac{\partial \bbP_{y(t)}}{\partial \hat{Y}_{ji}}\right)G^2(t)(\bbI-\bbP_{x(t)})\right]_{ji}\\
		\!-\!&\left[\boldsymbol{\Lambda}'\bbQ(t)G^2(t)\left(\frac{\partial \bbP_{x(t)}}{\partial \hat{Y}_{ji}}\!+\!\frac{\partial \bbP_{y(t)}}{\partial \hat{Y}_{ji}}\right)G(t)(\bbI-\bbP_{x(t)})\right]_{ji}\!-\!\left[\boldsymbol{\Lambda}'\bbQ(t)G^2(t)\frac{\partial\bbP_{x(t)}}{\partial \hat{Y}_{ji}}\right]_{ji}.
	\end{align*}
	We need some matrix derivative results
	\begin{align*}
		\frac{\partial \bbX(t)}{\partial \hat{Y}_{ji}}&=\sqrt{1-t}\boldsymbol{\Lambda}1_{ji},\  \frac{\partial \bbX'(t)}{\partial \hat{Y}_{ji}}=\sqrt{1-t}1_{ij}\boldsymbol{\Lambda}',\  \frac{\partial \bbX(t)}{\partial {Y}_{ji}}=\sqrt{t}\boldsymbol{\Lambda}1_{ji},\ \frac{\partial\bbX'(t)}{\partial {Y}_{ji}}=\sqrt{t}1_{ij}\boldsymbol{\Lambda}',\\
		\frac{\partial \bbQ(t)}{\partial \hat{Y}_{ji}}&=-(\bbX(t)\bbX'(t))^{-1}(\frac{\partial\bbX(t)}{\partial \hat{Y}_{ji}}\bbX'(t)+\bbX(t)\frac{\partial\bbX'(t)}{\partial \hat{Y}_{ji}})\bbQ(t)+(\bbX(t)\bbX'(t))^{-1}\frac{\partial\bbX(t)}{\partial \hat{Y}_{ji}}\\
		&=\!-\!\sqrt{1\!-\!t}(\bbX(t)\bbX'(t))^{-1}(\boldsymbol{\Lambda}1_{ji}\bbX'(t)\!+\!\bbX(t)1_{ij}\boldsymbol{\Lambda}')\bbQ(t)\!+\!\sqrt{1\!-\!t}(\bbX(t)\bbX'(t))^{-1}\boldsymbol{\Lambda}1_{ji}\\
		&=\sqrt{1-t}[(\bbX(t)\bbX'(t))^{-1}\boldsymbol{\Lambda}1_{ji}(\bbI-\bbP_{x(t)})-\bbQ(t)1_{ij}\boldsymbol{\Lambda}'\bbQ(t)],\\
		\frac{\partial \bbQ(t)}{\partial {Y}_{ji}}&=\sqrt{t}[(\bbX(t)\bbX'(t))^{-1}\boldsymbol{\Lambda}1_{ji}(\bbI-\bbP_{x(t)})-\bbQ(t)1_{ij}\boldsymbol{\Lambda}'\bbQ(t)],
	\end{align*}
	and
	\begin{align*}
		\frac{\partial\bbP_{x(t)}}{\partial \hat{Y}_{ji}}=&\frac{\partial \bbX(t)}{\partial \hat{Y}_{ji}}\bbQ(t)+\bbX(t)\frac{\partial \bbQ(t)}{\partial \hat{Y}_{ji}}\\
		=&\sqrt{1-t}1_{ij}\boldsymbol{\Lambda}'\bbQ(t)+\sqrt{1-t}[\bbQ'(t)\boldsymbol{\Lambda}1_{ji}(\bbI-\bbP_{x(t)})-\bbP_{x(t)}1_{ij}\boldsymbol{\Lambda}'\bbQ(t)]\\
		=&\sqrt{1-t}(\bbI-\bbP_{x(t)})1_{ij}\boldsymbol{\Lambda}'\bbQ(t)+\sqrt{1-t}\bbQ'(t)\boldsymbol{\Lambda}1_{ji}(\bbI-\bbP_{x(t)}),\\
		\frac{\partial\bbP_{x(t)}}{\partial {Y}_{ji}}=&\sqrt{t}(\bbI-\bbP_{x(t)})1_{ij}\boldsymbol{\Lambda}'\bbQ(t)+\sqrt{t}\bbQ'(t)\boldsymbol{\Lambda}1_{ji}(\bbI-\bbP_{x(t)}).
	\end{align*}
	Moreover, for $l=2$, we need the following results:
	\begin{align*}
		\frac{\partial \bbU(t)}{\partial {\hat{Y}}_{ji}}&=\sqrt{1-t}[(\bbY(t)\bbY'(t))^{-1}1_{ji}(\bbI-\bbP_{y(t)})-\bbU(t)1_{ij}\bbU(t)],\\
		\frac{\partial \bbU(t)}{\partial {Y}_{ji}}&=\sqrt{t}[(\bbY(t)\bbY'(t))^{-1}1_{ji}(\bbI-\bbP_{y(t)})-\bbU(t)1_{ij}\bbU(t)],\\
		\frac{\partial\bbP_{y(t)}}{\partial {\hat{Y}}_{ji}}&=\sqrt{1-t}(\bbI-\bbP_{y(t)})1_{ij}\bbU(t)+\sqrt{1-t}\bbU'(t)1_{ji}(\bbI-\bbP_{y(t)}),\\		\frac{\partial\bbP_{y(t)}}{\partial {Y}_{ji}}&=\sqrt{t}(\bbI-\bbP_{y(t)})1_{ij}\bbU(t)+\sqrt{t}\bbU'(t)1_{ji}(\bbI-\bbP_{y(t)}).
	\end{align*}
	Observed the difference between the derivative terms $\frac{\partial \bbQ(t)}{\partial {Y}_{ji}}$, $\frac{\partial \bbP_{x(t)}}{\partial {Y}_{ji}}$, $\frac{\partial \bbU}{\partial {Y}_{ji}}$, $\frac{\partial \bbP_{y(t)}}{\partial {Y}_{ji}}$ and  $\frac{\partial \bbQ(t)}{\partial \hat{Y}_{ji}}$, $\frac{\partial \bbP_{x(t)}}{\partial \hat{Y}_{ji}}$, $\frac{\partial \bbU}{\partial \hat{Y}_{ji}}$, $\frac{\partial \bbP_{y(t)}}{\partial \hat{Y}_{ji}}$ lies solely in their coefficients, which can be canceled out by the coefficients in formula \eqref{formula2}.
\end{proof}	
	
	\begin{lemma}\label{lemma4}
	Under the assumptions and notations of Theorem \ref{th12}, we calculate $\bbP_{x}-\bbP_{x}^{(k)}+\bbP_{y}-\bbP_{y}^{(k)}=\boldsymbol{\alpha}_k\boldsymbol{\beta}_k$, and obtain
	\begin{align*}
	\boldsymbol{\alpha}_k=&(\frac{e_k}{1+\gamma_k},\frac{\chi_k}{1+\gamma_k},\frac{e_k}{1+\gamma_k},-\frac{\chi_k}{1+\gamma_k},\frac{e_k}{1+\eta_k},\frac{\omega_k}{1+\eta_k},\frac{e_k}{1+\eta_k},-\frac{\omega_k}{1+\eta_k})\\
		\boldsymbol{\beta}_k=&(\chi_k,e_k,\gamma_ke_k,\chi_k,\omega_k,e_k,\eta_ke_k,\omega_k)'
	\end{align*}
	where 
	\begin{align*}
		\gamma_k=&(\boldsymbol{\Lambda} y_k+\boldsymbol{\Gamma}w_k)'[\bbX^{(k)}(\bbX^{(k)})']^{-1}(\boldsymbol{\Lambda} y_k+\boldsymbol{\Gamma}w_k), \quad
		\eta_k=y_k'[\bbY^{(k)}(\bbY^{(k)})']^{-1}y_k,\\
		\chi_k=&(\bbX^{(k)})'[\bbX^{(k)}(\bbX^{(k)})']^{-1}(\boldsymbol{\Lambda} y_k+\boldsymbol{\Gamma}w_k),\quad\quad \omega_k=(\bbY^{(k)})'[\bbY^{(k)}(\bbY^{(k)})']^{-1}y_k.
	\end{align*}
	\end{lemma}
	\begin{proof}
	By the formula (0.7.4.1) in \cite{HJ2012M}, we have 
		\begin{align*}
		&\bbP_{x}-\bbP_{x}^{(k)}=\bbX'(\bbX\bbX')^{-1}\bbX-(\bbX^{(k)})'[\bbX^{(k)}(\bbX^{(k)})']^{-1}\bbX^{(k)}\\
		=&(\boldsymbol{\Lambda} y_ke_k'+\boldsymbol{\Gamma}w_ke_k')'(\bbX\bbX')^{-1}\bbX+(\bbX^{(k)})'[(\bbX\bbX')^{-1}-(\bbX^{(k)}(\bbX^{(k)})')^{-1}]\bbX\\
		&+(\bbX^{(k)})'[\bbX^{(k)}(\bbX^{(k)})']^{-1}(\boldsymbol{\Lambda}y_ke_k'+\boldsymbol{\Gamma}w_ke_k')\\
		=&(\boldsymbol{\Lambda} y_ke_k'+\boldsymbol{\Gamma}w_ke_k')'(\bbX\bbX')^{-1}\bbX^{(k)}+(\boldsymbol{\Lambda} y_ke_k'+\boldsymbol{\Gamma}w_ke_k')'(\bbX\bbX')^{-1}(\boldsymbol{\Lambda} y_ke_k'+\boldsymbol{\Gamma}w_ke_k')\\
		&+(\bbX^{(k)})'[	(\bbX\bbX')^{-1}\!-\!(\bbX^{(k)}(\bbX^{(k)})')^{-1}]\bbX^{(k)}+(\bbX^{(k)})'[\bbX^{(k)}(\bbX^{(k)})']^{-1}(\boldsymbol{\Lambda} y_ke_k'+\boldsymbol{\Gamma}w_ke_k')\\
		&+(\bbX^{(k)})'[(\bbX\bbX')^{-1}\!-\!(\bbX^{(k)}(\bbX^{(k)})')^{-1}](\boldsymbol{\Lambda} y_ke_k'\!+\!\boldsymbol{\Gamma}w_ke_k')\\
		=&(\boldsymbol{\Lambda} y_ke_k'\!+\!\boldsymbol{\Gamma}w_ke_k')'[\bbX^{(k)}(\bbX^{(k)})']^{-1}\bbX^{(k)}\!+\!(\boldsymbol{\Lambda} y_ke_k'\!+\!\boldsymbol{\Gamma}w_ke_k')'\left((\bbX\bbX')^{-1}\!-\![\bbX^{(k)}(\bbX^{(k)})']^{-1}\right)\bbX^{(k)}\\
		&+(\boldsymbol{\Lambda} y_ke_k'+\boldsymbol{\Gamma}w_ke_k')'[\bbX^{(k)}(\bbX^{(k)})']^{-1}(\boldsymbol{\Lambda} y_ke_k'+\boldsymbol{\Gamma}w_ke_k')\\
		&+(\boldsymbol{\Lambda} y_ke_k'+\boldsymbol{\Gamma}w_ke_k')'\left([\bbX\bbX']^{-1}-[\bbX^{(k)}(\bbX^{(k)})']^{-1}\right)(\boldsymbol{\Lambda} y_ke_k'+\boldsymbol{\Gamma}w_ke_k')\\
		&+\!(\bbX^{(k)})'\left((\bbX\bbX')^{-1}\!-\![\bbX^{(k)}(\bbX^{(k)})']^{-1}\right)\bbX^{(k)}\!+\!(\bbX^{(k)})'[\bbX^{(k)}(\bbX^{(k)})']^{-1}(\boldsymbol{\Lambda} y_ke_k'\!+\!\boldsymbol{\Gamma}w_ke_k')\\
		&+(\bbX^{(k)})'\left((\bbX\bbX')^{-1}-[\bbX^{(k)}(\bbX^{(k)})']^{-1}\right)(\boldsymbol{\Lambda} y_ke_k'+\boldsymbol{\Gamma}w_ke_k').
	\end{align*}
		Here we need the following results:
	\begin{align*}
		\bbX\bbX'-\bbX^{(k)}(\bbX^{(k)})'=&x_kx_k'=(\boldsymbol{\Lambda} y_k+\boldsymbol{\Gamma}w_k)(\boldsymbol{\Lambda}y_k+\boldsymbol{\Gamma}w_k)',\\
		(\bbX\bbX')^{-1}\!-\![\bbX^{(k)}(\bbX^{(k)})']^{-1}=&\!-\!\frac{[\bbX^{(k)}(\bbX^{(k)})']^{-1}(\boldsymbol{\Lambda} y_k\!+\!\boldsymbol{\Gamma}w_k)(\boldsymbol{\Lambda}y_k\!+\!\boldsymbol{\Gamma}w_k)'[\bbX^{(k)}(\bbX^{(k)})']^{-1}}{1\!+\!\gamma_k}.
	\end{align*}
	The we have 
	\begin{align*}
		\bbP_{x}-\bbP_{x}^{(k)}=&(\boldsymbol{\Lambda} y_ke_k'+\boldsymbol{\Gamma}w_ke_k')'[\bbX^{(k)}(\bbX^{(k)})']^{-1}\bbX^{(k)}+\gamma_k e_ke_k'\\
		&-\frac{\gamma_k}{1+\gamma_k}(\boldsymbol{\Lambda} y_ke_k'+\boldsymbol{\Gamma}w_ke_k')'[\bbX^{(k)}(\bbX^{(k)})']^{-1}\bbX^{(k)}\\
		&-\frac{\gamma_k^2}{1+\gamma_k}e_ke_k'+(\bbX^{(k)})'[\bbX^{(k)}(\bbX^{(k)})']^{-1}(\boldsymbol{\Lambda} y_ke_k'+\boldsymbol{\Gamma}w_ke_k')\\
		&-\!\frac{1}{1\!+\!\gamma_k}(\bbX^{(k)})'(\bbX^{(k)}(\bbX^{(k)})')^{-1}(\boldsymbol{\Lambda} y_k\!+\!\boldsymbol{\Gamma}w_k)(\boldsymbol{\Lambda}y_k\!+\!\boldsymbol{\Gamma}w_k)'(\bbX^{(k)}(\bbX^{(k)})')^{-1}\bbX^{(k)}\\
		&-\frac{\gamma_k}{1+\gamma_k}(\bbX^{(k)})'(\bbX^{(k)}(\bbX^{(k)})')^{-1}(\boldsymbol{\Lambda} y_k+\boldsymbol{\Gamma}w_k)e_k'\\
	=&\frac{1}{1+\gamma_k}(e_k\chi_k'+\chi_ke_k')+\frac{\gamma_k}{1+\gamma_k}e_ke_k'-\frac{1}{1+\gamma_k}\chi_k\chi_k',
	\end{align*}
	abbreviated $(\bbX^{(k)})'[\bbX^{(k)}(\bbX^{(k)})']^{-1}(\boldsymbol{\Lambda} y_k+\boldsymbol{\Gamma}w_k)$ as $\chi_k$.
	Similarly,
	\begin{align*}
		\bbP_{y}-\bbP_{y}^{(k)}=& e_ky'_k[\bbY^{(k)}(\bbY^{(k)})']^{-1}\bbY^{(k)}\\
		&-e_ky'_k\frac{(\bbY^{(k)}(\bbY^{(k)})')^{-1}y_ky'_k(\bbY^{(k)}(\bbY^{(k)})')^{-1}}{1+y'_k(\bbY^{(k)}(\bbY^{(k)})')^{-1}y_k}\bbY^{(k)}\\
		&+e_ky'_k[\bbY^{(k)}(\bbY^{(k)})']^{-1}y_ke_k'\\
		&-e_ky'_k\frac{(\bbY^{(k)}(\bbY^{(k)})')^{-1}y_ky'_k(\bbY^{(k)}(\bbY^{(k)})')^{-1}}{1+y'_k(\bbY^{(k)}(\bbY^{(k)})')^{-1}y_k}y_ke_k'\\
		&-(\bbY^{(k)})'\frac{(\bbY^{(k)}(\bbY^{(k)})')^{-1}y_ky'_k(\bbY^{(k)}(\bbY^{(k)})')^{-1}}{1+y'_k(\bbY^{(k)}(\bbY^{(k)})')^{-1}y_k}\bbY^{(k)}\\
		&-(\bbY^{(k)})'\frac{(\bbY^{(k)}(\bbY^{(k)})')^{-1}y_ky'_k(\bbY^{(k)}(\bbY^{(k)})')^{-1}}{1+y'_k(\bbY^{(k)}(\bbY^{(k)})')^{-1}y_k}y_ke_k'\\		
		&+(\bbY^{(k)})'[\bbY^{(k)}(\bbY^{(k)})']^{-1}y_ke_k'\\
		=&\frac{1}{1+\eta_k}(e_k\omega_k'+\omega_ke_k')+\frac{\eta_k}{1+\eta_k}e_ke_k'-\frac{1}{1+\eta_k}\omega_k\omega_k',
	\end{align*}
	abbreviated $(\bbY^{(k)})'[\bbY^{(k)}(\bbY^{(k)})']^{-1}y_k$ as $\omega_k$.
	\end{proof}

\begin{lemma}\label{lemma5}
Under the assumptions and notations of Theorem \ref{th12}, for the matrix $\bbM$ with fixed dimension, we have
	\begin{align*}
		&\bbI_{8}+\boldsymbol{\beta}_k(\bbH^{(k)}-z)^{-1}\boldsymbol{\alpha}_k=\bbM+\tilde{\bbM}\\
		&(\bbI_{8}+\boldsymbol{\beta}_k(\bbH^{(k)}-z)^{-1}\boldsymbol{\alpha}_k)^{-1}=\bbM^{-1}+\widetilde{\bbM^{-1}},
	\end{align*}
	where for $l=1,2,4,$ $\mathbb{E}|\tilde{\bbM}_{ij}|^{2l}=O(\frac{1}{n})$ and for $r=1,2,$ $\mathbb{E}|(\widetilde{\bbM^{-1}})_{ij}|^{2r}=O(\frac{1}{n})$.
\end{lemma}
\begin{proof}
In order to calculate higher order moments, we truncate the entries of $\mathbf{Y}$ and $\mathbf{W}$ following the approach described in \cite{BaiS04C}, that is, $y_{ij}^{t}=y_{ij}I_{\{|y_{ij}|<\varepsilon_n\sqrt{n}\}}$. And 
		\begin{align*}
\|F^{\bbH}-F^{\bbH^{t}}\|\leq&\frac{1}{n}rank(\bbP_x-\bbP_{x^{t}}+\bbP_y-\bbP_{y^{t}})\\
\leq&\frac{c}{n}rank(\bbY-\bbY^{t})+\frac{c}{n}rank(\bbW-\bbW^{t})\\
\leq&\frac{c}{n}\sum_{i=1}^q\sum_{j=1}^nI_{\{|y_{ij}|\geq\varepsilon_n\sqrt{n}\}}+\frac{c}{n}\sum_{i=1}^p\sum_{j=1}^nI_{\{|w_{ij}|\geq\varepsilon_n\sqrt{n}\}},
		\end{align*}
		then we have 
		\begin{align*}
		\mathbb{P}(\|F^{\bbH}-F^{\bbH^{t}}\|\!\geq\!\delta)\!\leq\!\mathbb{P}(\frac{c}{n}\sum_{i=1}^q\sum_{j=1}^nI_{\{|y_{ij}|\geq\varepsilon_n\sqrt{n}\}}\!\geq\!\frac{\delta}{2})\!+\!\mathbb{P}(\frac{c}{n}\sum_{i=1}^p\sum_{j=1}^nI_{\{|w_{ij}|\geq\varepsilon_n\sqrt{n}\}}\!\geq\!\frac{\delta}{2}).
		\end{align*}
		According to concentration inequality and step 2 in \cite{2012The}, we have $\|F^{\bbH}-F^{\bbH^{t}}\|\overset{a.s.}{\to} 0$.
		Define $y_{ij}^{c}=y_{ij}^{t}-\mathbb{E}y_{ij}^{t}$, we consider
		\begin{align*}
		\|F^{\bbH^{t}}-F^{\bbH^{c}}\|\leq&\frac{1}{n}rank(\bbP_{x^{t}}-\bbP_{x^{c}}+\bbP_{y^{t}}-\bbP_{y^{c}})\\
\leq&\frac{c}{n}rank(\bbY^{t}-\bbY^{c})+\frac{c}{n}rank(\bbW^{t}-\bbW^{c})\leq\frac{c}{n}
		\end{align*}
		where for any $i,j$, $\{\mathbb{E}y_{ij}^{t}\}$ and $\{\mathbb{E}w_{ij}^{t}\}$ are the same, respectively. That is $\mathbb{E}\bbY^t$ and $\mathbb{E}\bbW^t$ are both rank one.
		Next write $\sigma_y=\mathbb{E}|y_{ij}^{c}|^2$, $y_{ij}^{s}=y_{ij}^{c}/\sigma_y$, we consider
		\begin{align*}
		L^3(F^{\bbH^{s}},F^{\bbH^{c}})\leq\tr(\bbP_{x^{s}}-\bbP_{x^{c}}+\bbP_{y^{s}}-\bbP_{y^{c}})(\bbP_{x^{s}}-\bbP_{x^{c}}+\bbP_{y^{s}}-\bbP_{y^{c}})
		\end{align*}
	noticing $\bbP_{y^{c}}$ is scale invariant, that is $\bbP_{y^{s}}=\bbP_{y^{c}}$. Then we have 
	\begin{align*}	
	L^3(F^{\bbH^{s}},F^{\bbH^{c}})\leq\tr(\bbP_{x^{s}}-\bbP_{x^{c}})(\bbP_{x^{s}}-\bbP_{x^{c}})\leq c\left(\frac{\sigma_w}{\sigma_y}-1\right)
		\end{align*}
		where 
		\begin{align*}
		&\bbP_{x^{s}}\!=\!\left(\boldsymbol{\Lambda}\frac{\bbY^{c}}{\sigma_y}\!+\!\boldsymbol{\Gamma}\frac{\bbW^{c}}{\sigma_w}\right)'\left[\left(\boldsymbol{\Lambda}\frac{\bbY^{c}}{\sigma_y}\!+\!\boldsymbol{\Gamma}\frac{\bbW^{c}}{\sigma_w}\right)\left(\boldsymbol{\Lambda}\frac{\bbY^{c}}{\sigma_y}\!+\!\boldsymbol{\Gamma}\frac{\bbW^{c}}{\sigma_w}\right)'\right]^{-1}\left(\boldsymbol{\Lambda}\frac{\bbY^{c}}{\sigma_y}\!+\!\boldsymbol{\Gamma}\frac{\bbW^{c}}{\sigma_w}\right)\\
		\!=\!&\left(\frac{\sigma_w}{\sigma_y}\boldsymbol{\Lambda}\bbY^{c}\!+\!\boldsymbol{\Gamma}\bbW^{c}\right)'\left[\left(\frac{\sigma_w}{\sigma_y}\boldsymbol{\Lambda}\bbY^{c}\!+\!\boldsymbol{\Gamma}\bbW^{c}\right)\left(\frac{\sigma_w}{\sigma_y}\boldsymbol{\Lambda}\bbY^{c}\!+\!\boldsymbol{\Gamma}\bbW^{c}\right)'\right]^{-1}\left(\frac{\sigma_w}{\sigma_y}\boldsymbol{\Lambda}\bbY^{c}\!+\!\boldsymbol{\Gamma}\bbW^{c}\right)
		\end{align*}
	which equals $\bbX^{s}-\bbX^{c}=(\frac{\sigma_w}{\sigma_y}-1)\boldsymbol{\Lambda}\bbY^{c}$.	
		Because the truncation is the same with that of \cite{BaiS04C}, obviously, we have $\sigma_y\to 1$ and $\sigma_w\to 1$, which implies $\sigma_w/\sigma_y-1\to 0$. In summary, truncation, centering and normalization do not change the limit of target LSD, so we assume the entries of $\bbW$ and $\bbY$ are bounded by $\varepsilon_n\sqrt{n}$ from now on.

	We first calculate the expectation for term $\boldsymbol{\beta}_k(\bbH^{(k)}-z)^{-1}\boldsymbol{\alpha}_k$ to approximate it. Here we mainly apply the following equalities:
	\begin{align*}
	&\mathbb{E}\chi_k'(\bbH^{(k)}-z)^{-1}\chi_k=\tr(\bbH^{(k)}-z)^{-1}(\bbQ^{(k)})'\bbQ^{(k)}:=\gamma_k^{\Lambda,H},\\
	&\mathbb{E}\chi_k'(\bbH^{(k)}-z)^{-1}\omega_k=\tr\boldsymbol{\Lambda}'\bbQ^{(k)}(\bbH^{(k)}-z)^{-1}(\bbU^{(k)})':=\theta_k^{\Lambda},\\
		&\mathbb{E}\omega_k'(\bbH^{(k)}-z)^{-1}\omega_k=\tr\bbU^{(k)}(\bbH^{(k)}-z)^{-1}(\bbU^{(k)})':=\eta_k^{H},\\
		&\mathbb{E}e_k'(\bbH^{(k)}-z)^{-1}e_k=(\bbH^{(k)}-z)^{-1}_{kk}:=\theta_k,\\
		&\mathbb{E}\gamma_k=\tr[\bbX^{(k)}(\bbX^{(k)})']^{-1}:=\gamma_k^0,\quad \mathbb{E}\eta_k=\tr[\bbY^{(k)}(\bbY^{(k)})']^{-1}:=\eta_k^0,
	\end{align*}
	then we have 
	\begin{align*}
	\bbM\!=\!\begin{pmatrix}
	1,&\frac{\gamma_k^{\Lambda,H}}{1+\gamma_k^0},&0,&-\frac{\gamma_k^{\Lambda,H}}{1+\gamma_k^0},&0,&\frac{\theta_k^{\Lambda}}{1+\eta_k^0},&0,&-\frac{\theta_k^{\Lambda}}{1+\eta_k^0}\\
	\frac{\theta_k}{1+\gamma_k^0},&1,&\frac{\theta_k}{1+\gamma_k^0},&0,&\frac{\theta_k}{1+\eta_k^0}, &0,&\frac{\theta_k}{1+\eta_k^0}, &0\\
		\frac{\gamma_k^0\theta_k}{1+\gamma_k^0},&0,&1+\frac{\gamma_k^0\theta_k}{1+\gamma_k^0},&0,&\frac{\gamma_k^0\theta_k}{1+\eta_k^0}, &0,&\frac{\gamma_k^0\theta_k}{1+\eta_k^0}, &0\\
		0,&\frac{\gamma_k^{\Lambda,H}}{1+\gamma_k^0},&0,&1-\frac{\gamma_k^{\Lambda,H}}{1+\gamma_k^0},&0,&\frac{\theta_k^{\Lambda}}{1+\eta_k^0},&0,&-\frac{\theta_k^{\Lambda}}{1+\eta_k^0}\\
	0,&\frac{\theta_k^{\Lambda}}{1+\gamma_k^0}, &0, &-\frac{\theta_k^{\Lambda}}{1+\gamma_k^0}, &1, &\frac{\eta_k^H}{1+\eta_k^0}, &0, &-\frac{\eta_k^H}{1+\eta_k^0}\\
	\frac{\theta_k}{1+\gamma_k^0},&0,&\frac{\theta_k}{1+\gamma_k^0},&0,&\frac{\theta_k}{1+\eta_k^0}, &1,&\frac{\theta_k}{1+\eta_k^0}, &0\\
		\frac{\eta_k^0\theta_k}{1+\gamma_k^0},&0,&\frac{\eta_k^0\theta_k}{1+\gamma_k^0},&0,&\frac{\eta_k^0\theta_k}{1+\eta_k^0}, &0,&1+\frac{\eta_k^0\theta_k}{1+\eta_k^0}, &0\\
	0,&\frac{\theta_k^{\Lambda}}{1+\gamma_k^0}, &0, &-\frac{\theta_k^{\Lambda}}{1+\gamma_k^0}, &0, &\frac{\eta_k^H}{1+\eta_k^0}, &0, &1-\frac{\eta_k^H}{1+\eta_k^0}\\
	\end{pmatrix}.
	\end{align*}
	Obviously, we have for $l=1,2,4,$
	\begin{align*}
	\mathbb{E}|1\!+\![\boldsymbol{\beta}_k(\bbH^{(k)}\!-\!z)^{-1}\boldsymbol{\alpha}_k]_{ij}\!-\!\bbM_{ij}|^{2l}=O\left(\frac{1}{n}\right).
	\end{align*}
	Combining the formula $A^{-1}\!-\!B^{-1}=\!-\!B^{-1}(A\!-\!B)A^{-1}=\!-\!B^{-1}(A\!-\!B)B^{-1}\!+\!B^{-1}(A\!-\!B)B^{-1}(A\!-\!B)A^{-1}$, we have
	\begin{align*}
	\mathbb{E}|\widetilde{\bbM^{-1}}_{ij}|^2=C\mathbb{E}|\tilde{\bbM}_{ij}|^2+C\sqrt{\mathbb{E}|\tilde{\bbM}_{ij}|^4}\sqrt{\mathbb{E}|\tilde{\bbM}_{i'j'}|^4}=O\left(\frac{1}{n}\right),\\
	\mathbb{E}|\widetilde{\bbM^{-1}}_{ij}|^4=C\mathbb{E}|\tilde{\bbM}_{ij}|^4+C\sqrt{\mathbb{E}|\tilde{\bbM}_{ij}|^8}\sqrt{\mathbb{E}|\tilde{\bbM}_{i'j'}|^8}=O\left(\frac{1}{n}\right),
	\end{align*}
	where we need the entrices in $[\bbI_8+\boldsymbol{\beta}_k(\bbH-z)^{-1}\boldsymbol{\alpha}_k]^{-1}$ and $\bbM^{-1}$ are bounded with probability 1.
	
	Consider the matrix $\bbM-\bbI$ has three linearly independent rows, its rank is $3$. We calculate the imaginary part of non-zero eigenvalues $\lambda$ via the fact that two  determinants of $4\times 4$ minor of $\Im \bbM-\Im \lambda\bbI$ are zero, which are
	\begin{align*}
	\det \begin{pmatrix}
	-\Im\lambda, &\frac{\Im\theta_k}{1+\gamma_k^0}, &\frac{\Im\theta_k}{1+\eta_k^0},&0\\
	\gamma_k^0\Im\lambda, &-\Im\lambda, &0, &0\\
	\eta_k^0\Im\lambda, &0, &-\Im\lambda,  &0\\
	\Im\lambda, &0, &0, &-\Im\lambda
	\end{pmatrix}=0=\det \begin{pmatrix}
	0, &0, &-\frac{\Im\theta_k^{\Lambda}}{1+\gamma_k^0}, &-\Im\lambda-\frac{\Im\eta_k^{H}}{1+\eta_k^0}\\
		0, &-\Im\lambda, &-\frac{\Im\gamma_k^{\Lambda,H}}{1+\gamma_k^0}, &-\frac{\Im\theta_k^{\Lambda}}{1+\eta_k^0}\\
		0, &\Im\lambda, &-\Im\lambda, &0\\
		-\Im\lambda, &0, &0, &\Im\lambda
	\end{pmatrix}.
	\end{align*}
	The we have $\Im\lambda$ satisfies the following equations
	\begin{align*}
	&(\Im\lambda)^3\left[\Im\lambda-\left(\frac{\gamma_k^0}{1+\gamma_k^0}-\frac{\eta_k^0}{1+\eta_k^0}\right)\Im\theta_k\right]=0\\
		&(\Im\lambda)^2\left[(\Im\lambda)^2+\left(\frac{\Im\gamma_k^{\Lambda,H}}{1+\gamma_k^0}+\frac{\Im\eta_k^{H}}{1+\eta_k^0}\right)\Im\lambda-\frac{(\Im\theta_k^{\Lambda})^2-\Im\gamma_k^{\Lambda,H}\Im\eta_k^{H}}{(1+\gamma_k^0)(1+\eta_k^0)}\right]=0,
	\end{align*}
	where $\gamma_k^0$ and $\eta_k^0$ are positive.
	Because when $\Im z>0$, $\Im\theta_k>0$, we have $\Im\lambda_1\neq 0$. Notice that the formula $|\tr(AB')|^2\leq \tr(AA')\tr(BB')$, and  when $A=cB$, the equality holds.
	Combining the following facts
	\begin{align*}
	\Im\theta_k^{\Lambda}=&(\Im z)\tr\boldsymbol{\Lambda}'\bbQ^{(k)}(\bbH^{(k)}-z)^{-1}(\bbH^{(k)}-\bar{z})^{-1}(\bbU^{(k)})'\\
		\Im\gamma_k^{\Lambda,H}=&(\Im z)\tr\bbQ^{(k)}(\bbH^{(k)}-z)^{-1}(\bbH^{(k)}-\bar{z})^{-1}(\bbQ^{(k)})'\\
		\Im\eta_k^{H}=&(\Im z)\tr\bbU^{(k)}(\bbH^{(k)}-z)^{-1}(\bbH^{(k)}-\bar{z})^{-1}(\bbU^{(k)})'
	\end{align*}
	we have $(\Im\theta_k^{\Lambda})^2-\Im\gamma_k^{\Lambda,H}\Im\eta_k^{H}<0$ strictly holds, then $\Im\lambda_2\neq 0$ and $\Im\lambda_3\neq 0$. This implies the non-zero eigenvalues of $\bbM-\bbI$ are complex, away from $-1$. 
	\end{proof}

\section*{Acknowledgments}

The author is very grateful to Zhidong Bai for many stimulate conversations concerning the proof of Lemmas. 
The author also thanks Jiang Hu for inspiring to explore the relationship (Theorem \ref{th11}).

\bibliographystyle{imsart-number} 
\bibliography{paper}       


\end{document}